\documentclass[11pt]{article}
\usepackage{mathrsfs}
\usepackage{amsmath,amsthm,amssymb}
\usepackage{amsfonts}
\usepackage{color,xcolor}
\usepackage{graphicx}
\usepackage{geometry}
\usepackage{manfnt}
\usepackage{booktabs}
\usepackage{siunitx}
\usepackage{multirow}
\usepackage{subcaption}
\usepackage{float}    % Ìá¹© [H]
\usepackage[section]{placeins}
\usepackage[pagewise]{lineno}%\linenumbers
\usepackage[colorlinks,citecolor=blue,urlcolor=blue]{hyperref}
\allowbreak
\newtheorem{theorem}{Theorem}[section]

\newtheorem{lemma}[theorem]{Lemma}

\newtheorem{definition}[theorem]{Definition}
\newtheorem{remark}[theorem]{Remark}
\newtheorem{example}[theorem]{Example}

\def\cC{{\mathcal C}}
\def\cE{{\mathcal E}}
\def\cK{{\mathcal K}}

\def\geq{\geqslant}

\begin{document}
\title{ \bf Strong solutions and sharp Euler--Maruyama approximations for SDEs with Lebesgue--Dini drift}
\author{Jinlong Wei$^a$, Junhao Hu$^b$, Guangying Lv$^c$ and Chenggui Yuan$^d$\thanks{Corresponding author.}
\\{\small \it $^a$School of Statistics and Mathematics, Zhongnan University of Economics}\\
{\small \it  and Law, Wuhan 430073, China}  \\ {\small \tt  weijinlong.hust@gmail.com}
\\ {\small \it $^b$School of Mathematics and Statistics, South-Central Minzu University} \\ {\small \it Wuhan 430074, China} \\ {\small \tt junhaohu74@163.com}
\\ {\small \it $^c$College of Mathematics and Statistics,
Nanjing University of Information}
\\{\small \it  Science and Technology, Nanjing 210044, China}\\ {\small \tt  gylvmaths@126.com}
\\ {\small \it $^d$Department of Mathematics, Swansea University, Bay Campus} \\ {\small \it Swansea SA1 8EN, United Kingdom} \\ {\small \tt C.Yuan@swansea.ac.uk}
}

\date{}

\maketitle

\noindent{\hrulefill}
\vskip1mm\noindent{\bf Abstract} We investigate the strong approximation of stochastic differential equations whose drift is square-integrable in time and Dini continuous in space, while the diffusion coefficient is non-constant and uniformly elliptic. Using a refined It\^{o}--Tanaka trick combined with parabolic regularity estimates, we first establish strong well-posedness and the stochastic flow property.
Under additional Lipschitz regularity of the diffusion matrix, we then analyze a polygonal-type Euler--Maruyama scheme and prove the strong error estimate
\[
\Big\|\sup_{0\le t\le1}|X_t-X_t^n|\Big\|_{L^p(\Omega)}
\le C
n^{-\frac12}\log(n)^{\frac32}, \quad p\ge2.
\]
We further show that this rate is sharp: even under smooth and uniformly elliptic diffusion coefficients with vanishing drift, the convergence order $1/2$ cannot be improved. These results provide the first sharp quantitative strong convergence estimates in a Lebesgue--Dini drift framework.



 \vskip0mm\noindent{\hrulefill}

\section{Introduction}\label{sec1}\setcounter{equation}{0}
We consider the stochastic differential equation (SDE) in $\mathbb{R}^d$
\begin{equation}\label{1.1}
dX_{s,t}(x)=b(t,X_{s,t}(x))dt+\sigma(t,X_{s,t}(x))dW_t,\quad t\in (s,1], \qquad
X_{s,t}(x)\big|_{t=s}=x\in\mathbb{R}^d,
\end{equation}
where $s\in [0,1)$,
$\{W_t\}_{0\le t\le 1}=\{(W_{1,t},\ldots,W_{d,t})^\top\}_{0\le t\le 1}$
is a $d$-dimensional standard Wiener process defined on a stochastic basis
$(\Omega,\mathcal{F},\mathbb{P},\{\mathcal{F}_t\}_{0\le t\le 1})$.
The coefficients $b:[0,1]\times\mathbb{R}^d\to\mathbb{R}^d$ and
$\sigma:[0,1]\times\mathbb{R}^d\to\mathbb{R}^{d\times d}$ are Borel measurable.

\smallskip
The classical theory ensures strong existence and pathwise uniqueness for \eqref{1.1}
under Lipschitz assumptions on the coefficients; see It\^{o}~\cite{Ito}.
A substantial body of work has subsequently shown that the Wiener process can
regularize the dynamics and allow for strong well-posedness far beyond the
Lipschitz regime. In particular, bounded measurable drifts were treated by
Veretennikov~\cite{Ver}, and further refinements and variants can be found in,
e.g., \cite{MNP,WLW}. A major step was made by Krylov and R\"ockner~\cite{KR},
who proved strong well-posedness in the integrable (subcritical) LPS class for
$\sigma=I_{d\times d}$:
\begin{equation}\label{1.2}
b\in L^q([0,1];L^p(\mathbb{R}^d;\mathbb{R}^d)),\quad p,q\in[2,\infty),\qquad
\tfrac{2}{q}+\tfrac{d}{p}<1.
\end{equation}
We refer to Zhang~\cite{Zha05,Zha11}, Zhang and Yuan~\cite{ZY}, Xia, Xie, Zhang and Zhao~\cite{XX},
and \cite{GG,WHY1} for related developments, including extensions to non-constant diffusion coefficients.

\smallskip
Motivated in part by the Navier--Stokes equations, it is natural to ask whether
strong well-posedness persists at the critical threshold where equality holds in \eqref{1.2};
see \cite{CI,Lady}. This critical problem was open for a long time after \cite{KR}.
Recently, R\"ockner and Zhao~\cite[Theorem~1.1]{RZ2} established strong well-posedness for $d\ge 3$
(with $\sigma=I_{d\times d}$) under the critical LPS condition
\begin{equation*}
b\in L^q([0,1];L^p(\mathbb{R}^d;\mathbb{R}^d)),\quad p,q\in(2,\infty),\qquad
\tfrac{2}{q}+\tfrac{d}{p}=1,
\end{equation*}
or under the continuity-in-time assumption
\begin{equation*}
b\in \mathcal{C}([0,1];L^d(\mathbb{R}^d;\mathbb{R}^d)).
\end{equation*}
Further contributions and complementary approaches can be found in
\cite{Kry21-2,Kry21-3,Kry23-2,Nam,WLWu,WHY2}.

\smallskip
The borderline case $q=2$, $p=\infty$ is particularly subtle.
For $\sigma=I_{d\times d}$, Beck, Flandoli, Gubinelli and Maurelli~\cite{BFGM} obtained
strong existence and pathwise uniqueness for a.e.\ starting point $x\in\mathbb{R}^d$;
see also \cite{KSS} for form-bounded drifts.
Wei, Wang, Lv and Duan~\cite{WWLD} later upgraded this to every $x\in\mathbb{R}^d$
under an additional local Dini continuity assumption on $b$.
More recently, Krylov~\cite{Kry25} proved strong uniqueness for Morrey drifts and VMO diffusion
coefficients, covering in particular the case $q=2$ and $p=\infty$,
while the strong existence problem has remained open in general.
One of the aims of the present paper is to provide a partially positive answer in this direction
and to show that \eqref{1.1} generates a unique stochastic flow of homeomorphisms.

\smallskip
The second part of this paper is devoted to quantitative approximation of \eqref{1.1}.
While Euler--Maruyama-type schemes are classical and well understood for smooth coefficients,
their analysis in low-regularity regimes involves several structural difficulties.
These difficulties stem not only from the irregularity of the drift, but also from
the interaction between temporal discretization and stochastic integration,
especially in the presence of a non-constant diffusion coefficient. Our main objective is to obtain strong $L^p$-error bounds with explicit rates in a borderline regime where well-posedness and numerical analysis meet.

\smallskip
We focus on two closely related discretization schemes for \eqref{1.1} (with $s=0$),
namely the Euler--Maruyama scheme
\begin{equation}\label{1.3}
	d\tilde{X}_t^n
	= b\big(\kappa_n(t),\tilde{X}_{\kappa_n(t)}^n\big)dt
	+ \sigma\big(\kappa_n(t),\tilde{X}_{\kappa_n(t)}^n\big)dW_t,
	\quad t\in(0,1], \quad \tilde{X}_0^n = x\in\mathbb{R}^d,
\end{equation}
and the polygonal-type approximation
\begin{equation}\label{1.4}
	dX_t^n
	= b\big(t,X_{\kappa_n(t)}^n\big)dt
	+ \sigma\big(t,X_{\kappa_n(t)}^n\big)dW_t,
	\quad t\in(0,1], \quad X_0^n = x\in\mathbb{R}^d,
\end{equation}
where $n\in\mathbb{N}$ and $\kappa_n(t)=\lfloor nt\rfloor/n$.
The key difference between the two schemes lies in the treatment of the drift and diffusion:
the polygonal scheme avoids evaluating $b$ and $\sigma$ at the discretized time argument,
which is crucial when $b$ and $\sigma$ have limited temporal regularity, since it avoids evaluating the coefficients at discontinuous time arguments and yields better compatibility with stochastic integration.

\smallskip
At the level of classical regularity, strong convergence properties of
Euler--Maruyama schemes are by now well understood.
If both $b$ and $\sigma$ are Lipschitz continuous in space and
$1/2$-H\"older continuous in time, then the strong $L^p$ convergence rate
is known to be $1/2$; see \cite{KP} for $d\ge1$ and \cite{Yan} for $d=1$.
Subsequent works relaxed these assumptions by weakening either the spatial
or the temporal regularity of the drift, while still retaining quantitative
rates; see, for instance, \cite{BHY,GR,NT1,Pamen}.
However, all these results rely on a degree of continuity that excludes
several drift classes arising naturally in the theory of singular SDEs.

\smallskip
A further difficulty arises when the drift is discontinuous in space
and/or time.
In such cases, the standard Euler--Maruyama scheme \eqref{1.3} may even
become ill-defined when the numerical trajectory approaches a singularity
of $b$, potentially leading to severe instability.
Strong convergence for multidimensional SDEs with discontinuous drift
and possibly degenerate diffusion coefficients has been established in
\cite{LS}, where a rate of order $1/4-\varepsilon$ for arbitrarily small $\varepsilon>0$ is obtained under piecewise
Lipschitz and geometric non-parallelity conditions.
This line of work highlights the intrinsic challenges posed by spatial
irregularities of the drift, even in time-independent settings.
These difficulties have motivated the development of various taming or
truncation procedures designed to restore numerical stability, among which
the tamed Euler--Maruyama method \cite{LL} has emerged as a widely used and
effective approach.

\smallskip
Beyond strong $L^p$-error estimates and stability considerations, a natural
question concerns convergence in distributional sense, such as total
variation distance or related notions.
Results in this direction are available, for example, for time-independent
drifts under integrability or H\"older-type assumptions \cite{SYZ}, and for
bounded, time-dependent coefficients \cite{BJ}.
Within the Krylov--R\"ockner framework, Jourdain and Menozzi \cite{JS}
established quantitative bounds on the marginal densities of tamed
Euler--Maruyama schemes.
More recently, L\^{e} and Ling \cite{LL} employed stochastic sewing
techniques to improve strong convergence rates under additional spatial
regularity assumptions on the diffusion coefficient.
Taken together, these works demonstrate the effectiveness of stochastic
sewing methods in the numerical analysis of singular SDEs, while leaving
several borderline regularity regimes unresolved.

\smallskip
In particular, strong error estimates remain largely open when the drift
is in $L^2([0,1];L^\infty(\mathbb{R}^d;\mathbb{R}^d))$, especially in the
presence of general, non-constant diffusion coefficients.
This regime lies at the interface between critical well-posedness theory
and numerical approximation, and cannot be adequately treated by existing
approaches relying solely on regularization or taming techniques.

\smallskip
The contribution of this paper is twofold.
First, using the It\^{o}--Tanaka trick, we establish strong well-posedness
and the stochastic flow property for \eqref{1.1} when the drift is
square-integrable in time, bounded and Dini continuous in space.
Second, under an additional Lipschitz assumption on the diffusion matrix, we prove a sharp strong convergence bound for the polygonal-type Euler--Maruyama scheme of order $n^{-1/2}(\log n)^{3/2}$.

\subsection{Setup and notation}\label{sec1.1}
We introduce the main notation and conventions adopted throughout the paper.

$\bullet$ $\mathbb{N}$ denotes the set of natural numbers, and $\mathbb{R}_+$ denotes the set of positive real numbers.
The symbol $\nabla$ stands for the gradient with respect to the spatial variables.

$\bullet$ For a $d\times d$ matrix $a=(a_{ij})_{1\le i,j\le d}$, the symbol $a_{ij}$ denotes its $(i,j)$-th entry, namely the element in the $i$-th row and $j$-th column.
For a vector $z=(z_1,\ldots,z_d)\in\mathbb{R}^d$, the symbol $z_i$ denotes its $i$-th component.

$\bullet$ A continuous function $\rho:[0,1]\to\mathbb{R}_+$ is said to be
slowly varying at zero (in the sense of Karamata~\cite[p.~6]{BGT}) if, for all $\upsilon>0$,
\[
\lim_{r\to 0}\frac{\rho(\upsilon r)}{\rho(r)}=1.
\]

$\bullet$ A continuous and increasing function $\rho:[0,1]\to\mathbb{R}_+$ is called a
Dini function if
\[
\int_0^1\frac{\rho(r)}{r}dr<\infty.
\]
Let $h:\mathbb{R}^d\to\mathbb{R}$ be a continuous function.
If there exists a constant $C>0$ such that
\[
|h(x)-h(y)|\le C\rho(|x-y|), \quad x,y\in\mathbb{R}^d,\ \ |x-y|\le1,
\]
then $h$ is said to be Dini continuous. If, in addition, $h$ is bounded, then we call it a bounded Dini function.
The space of all bounded Dini functions is denoted by $\mathcal{D}_b^\rho(\mathbb{R}^d)$.
For $h\in\mathcal{D}_b^\rho(\mathbb{R}^d)$, we define the norm
\[
\|h\|_{\mathcal{D}_b^\rho(\mathbb{R}^d)}
=\sup_{x\in\mathbb{R}^d}|h(x)|
+\sup_{0<|x-y|\le1}\frac{|h(x)-h(y)|}{\rho(|x-y|)}
=: \|h\|_0+[h]_\rho =: \|h\|_\rho,
\]
under which $\mathcal{D}_b^\rho(\mathbb{R}^d)$ is a Banach space.

$\bullet$ We denote by $L^2([0,1];\mathcal{D}_b^\rho(\mathbb{R}^d))$ (the bounded Lebesgue--Dini space)
the set of all Borel measurable functions
$h\in L^2([0,1];\mathcal{C}_b(\mathbb{R}^d))$
such that
\[
|h(t,x)-h(t,y)|\le f(t)\rho(|x-y|),
\quad x,y\in\mathbb{R}^d,\ \ |x-y|\le1,
\]
for some function $f\in L^2([0,1])$.
For $h\in L^2([0,1];\mathcal{D}_b^\rho(\mathbb{R}^d))$, we define the norm
\[
\|h\|_{L^2([0,1];\mathcal{D}_b^\rho(\mathbb{R}^d))}
=\bigg(\int_0^1\|h(t,\cdot)\|_\rho^2dt\bigg)^{\frac12}
=:\|h\|_{2,\rho}.
\]
Then $L^2([0,1];\mathcal{D}_b^\rho(\mathbb{R}^d))$ is a Banach space.
For simplicity, we write $\|h(t,\cdot)\|_\rho$ as $\|h(t)\|_\rho$ throughout the paper. Moreover, for a vector-valued function
$h=(h_1,\ldots,h_d):[0,1]\times\mathbb{R}^d\to\mathbb{R}^d$
with $h_j\in L^2([0,1];\mathcal{D}_b^\rho(\mathbb{R}^d))$ for $1\le j\le d$,
we say that $h\in L^2([0,1];\mathcal{D}_b^\rho(\mathbb{R}^d;\mathbb{R}^d))$ and define
\[
\|h\|_{2,\rho}
:=\bigg(\sum_{j=1}^d\|h_j\|_{2,\rho}^2\bigg)^{\frac12}.
\]

$\bullet$ Given a filtered probability space
$(\Omega,\mathcal{F},\{\mathcal{F}_t\}_{0\le t\le 1},\mathbb{P})$,
we always assume that the filtration satisfies the usual conditions;
in particular, $\mathcal{F}_0$ is complete.
For conditional expectations, we write
$\mathbb{E}_t Y := \mathbb{E}[Y|\mathcal{F}_t]$ and define
\[
\|Y\|_{L^p(\Omega|\mathcal{F}_t)}
:=\big(\mathbb{E}[|Y|^p|\mathcal{F}_t]\big)^{\frac1p}, \quad p\ge1.
\]

$\bullet$ Let $g\in L^p([0,1];L^\infty(\mathbb{R}^d))$
or $f\in L^p([0,1];\mathcal{C}_b(\mathbb{R}^d))$ with $p\in[2,\infty]$
and $0\le s\le t\le1$.
We use the notation
$\|g\|_{p,\infty,[s,t]}:=\|g\|_{L^p([s,t];L^\infty(\mathbb{R}^d))}$
and
$\|f\|_{p,0,[s,t]}:=\|f\|_{L^p([s,t];\mathcal{C}_b(\mathbb{R}^d))}$.
When $s=0$ and $t=1$, we simply write
$\|g\|_{p,\infty}$ and $\|f\|_{p,0}$. Similarly, for every $t\in [0,1]$, we use the notation $\|g(t)\|_\infty:=\|g(t,\cdot)\|_\infty$.

$\bullet$ For $0\le S\le T\le1$, define the simplices
\[
[S,T]_\le^2=\{(s,t):S\le s\le t\le T\},\qquad
[S,T]_\le^3=\{(s,u,t):S\le s\le u\le t\le T\}.
\]
Given a map $A:[S,T]_\le^2\to\mathbb{R}^d$, we define its increment
$\delta A:[S,T]_\le^3\to\mathbb{R}^d$ by
\[
\delta A_{s,u,t}:=A_{s,t}-A_{s,u}-A_{u,t}.
\]

$\bullet$ A measurable function $w:[0,1]_\le^2\to\mathbb{R}_+$ is called a
control if it is superadditive, that is,
\[
w(s,u)+w(u,t)\le w(s,t),\qquad (s,u,t)\in[0,1]_\le^3.
\]
If $w_1$ and $w_2$ are controls and $\theta_1,\theta_2\in\mathbb{R}_+$ satisfy
$\theta_1+\theta_2\ge1$, then $w:=w_1^{\theta_1}w_2^{\theta_2}$ is also a control.

$\bullet$ Throughout the paper, the letter $C$ denotes a generic positive constant
whose value may change from line to line.
For a parameter or function $\tilde{\zeta}$, the notation $C(\tilde{\zeta})$
indicates that the constant depends only on $\tilde{\zeta}$.
When no confusion arises, we simply write $C$.

\subsection{Main results}\label{sec1.2}

In this subsection, we present our main results.
The first one concerns the strong well-posedness.
Before giving the result, we introduce the following definition.

\begin{definition}[\cite{Kun90}, p.~114]\label{def1.1}
A stochastic flow of homeomorphisms on a stochastic basis
$(\Omega,\mathcal{F},\mathbb{P},$ $\{\mathcal{F}_t\}_{t\in[0,1]})$
associated with \eqref{1.1}
is a map
\[
(s,t,x,\omega)\mapsto X_{s,t}(x,\omega),
\quad 0\le s\le t\le1,\; x\in\mathbb{R}^d,\; \omega\in\Omega,
\]
with values in $\mathbb{R}^d$, such that:

\smallskip
\noindent (i)
For each $s\in[0,1]$ and $x\in\mathbb{R}^d$, the process
$\{X_{s,t}(x)\}_{s\le t\le 1}$ is a continuous
$\{\mathcal{F}_{s,t}\}_{s\le t\le1}$-adapted solution of \eqref{1.1}.

\smallskip
\noindent (ii)
$\mathbb{P}$-a.s., $X_{s,t}(\cdot)$ is a homeomorphism for all $0\le s\le t\le1$,
and both $X_{s,t}(x)$ and its inverse $X_{s,t}^{-1}(x)$ are continuous in $(s,t,x)$.

\smallskip
\noindent (iii)
$\mathbb{P}$-a.s., the flow property holds:
$X_{s,t}(x)=X_{r,t}(X_{s,r}(x))$
for all $0\le s\le r\le t\le1$ and $x\in\mathbb{R}^d$,
and $X_{s,s}(x)=x$.
\end{definition}

Our first main result is the following.
\begin{theorem}\label{the1.2}
Let
$b\in L^2([0,1];\mathcal{D}_b^\rho(\mathbb{R}^d;\mathbb{R}^d))$
such that $\rho$ is slowly varying at zero and $\rho^{1/2}$ is a Dini function.
Suppose that $\sigma\in L^2([0,1];W^{1,\infty}(\mathbb{R}^d;\mathbb{R}^{d\times d}))$
and that $a(t,x)=\sigma(t,x)\sigma^{T}(t,x)$ satisfies:

\smallskip
$\bullet$
There exists a constant $\Gamma>1$ such that
\begin{equation}\label{1.5}
\Gamma^{-1}|\xi|^2
\le \sum_{i,j=1}^d a_{i,j}(t,x)\xi_i\xi_j
\le \Gamma|\xi|^2,
\quad \forall \ (t,x)\in[0,1]\times \mathbb{R}^d,\ \xi\in\mathbb{R}^d.
\end{equation}

$\bullet$
There exists $\alpha\in(0,1]$ such that
\begin{equation}\label{1.6}
|a_{i,j}(t,x)-a_{i,j}(t,y)|\le C|x-y|^\alpha,
\quad \forall\ t\in[0,1], \ \ 1\le i,j\le d.
\end{equation}
Then there is a unique stochastic flow of homeomorphisms
$\{X_{s,t}(\cdot)\}_{0\le s\le t\le 1}$ associated with \eqref{1.1}.
\end{theorem}

We present an explicit example illustrating the scope of Theorem~\ref{the1.2}.
The drift satisfies a Lebesgue--Dini spatial regularity, strictly weaker than any
H\"older continuity, while the diffusion coefficient is non-constant, uniformly elliptic,
and spatially H\"older continuous.

\begin{example}\label{ex1.3}
Fix $\beta>2$ and define
\[
\rho(r):=\bigl(\log(e/r)\bigr)^{-\beta},\qquad r\in(0,1].
\]
Then $\rho$ is slowly varying at zero and $\rho^{1/2}$ is a Dini function. Let $\varphi:\mathbb{R}\to[0,1/2]$ be the $1$-periodic sawtooth function
\[
\varphi(\upsilon):=\min\bigl\{\mathrm{dist}(\upsilon,\mathbb{Z}),1-\mathrm{dist}(\upsilon,\mathbb{Z})\bigr\},
\qquad
\mathrm{dist}(\upsilon,\mathbb{Z})=\inf_{m\in\mathbb{Z}}|\upsilon-m|,\]
which is bounded and $1$-Lipschitz.
Define
\[
a_k:=\rho(2^{-k})-\rho(2^{-(k+1)})\ge0,
\qquad
g(x):=\sum_{k=1}^{\infty} a_k\varphi(2^k x).
\]
For any $f\in L^2([0,1])$, set
\[
b(t,x):=f(t)(g(x_1),0,\ldots,0),\qquad (t,x)\in[0,1]\times\mathbb{R}^d.
\]
Then the drift satisfies $b\in L^2\bigl([0,1];\mathcal{D}_b^\rho(\mathbb{R}^d;\mathbb{R}^d)\bigr)$,
while its spatial regularity is strictly weaker than any H\"older continuity, namely
\[
b(t,\cdot)\notin \mathcal{C}^\alpha_b(\mathbb{R}^d)
\quad\text{for all }\alpha\in(0,1).
\]

Let $\varepsilon\in(0,1/2)$ and define
\[
\sigma(t,x)
:= I_{d\times d}
+\varepsilon \sin(2\pi t)
\mathrm{diag}\bigl(\tanh(x_1),\ldots,\tanh(x_d)\bigr).
\]
Then $\sigma\in L^\infty\bigl([0,1];W^{1,\infty}(\mathbb{R}^d;\mathbb{R}^{d\times d})\bigr)$.
Moreover, the diffusion matrix $a=\sigma\sigma^T$ is uniformly elliptic:
\[
(1-\varepsilon)^2|\xi|^2
\le \xi^T a(t,x)\xi
\le (1+\varepsilon)^2|\xi|^2,
\qquad \forall \ (t,x)\in[0,1]\times\mathbb{R}^d,\ \xi\in\mathbb{R}^d.
\]
Consequently, all assumptions of Theorem~\ref{the1.2} are satisfied and there exists a unique stochastic flow of homeomorphisms for this SDE.
\end{example}

Based on the strong well-posedness established in Theorem \ref{the1.2}, we obtain the following convergence rate under additional regularity assumptions on the diffusion coefficient.
\begin{theorem}\label{the1.4}
Let $b$ be as in Theorem~\ref{the1.2}.
Assume that $a=\sigma\sigma^T$ satisfies \eqref{1.5}, and \eqref{1.6} with $\alpha=1$.
Let $X_t$ and $X_t^n$ be the unique strong solutions of \eqref{1.1} and \eqref{1.4},
respectively. Then, for every $p\ge 2$, and every integer $n\ge 2$, we have
\begin{equation}\label{1.7}
\Big\|\sup_{0\le t\le1}|X_t-X_t^n|\Big\|_{L^p(\Omega)}
\le C\!\left(p,d,\|a\|_{\infty,\infty},
\|\nabla a\|_{\infty,\infty},\|b\|_{2,0}\right)
n^{-\frac12}\log(n)^{\frac32}.
\end{equation}
\end{theorem}

The estimate \eqref{1.7} shows that, under the additional Lipschitz regularity
and uniform ellipticity on $a$,
the Euler--Maruyama scheme attains the strong convergence order $1/2$
(up to a logarithmic factor).
Since uniform ellipticity provides additional non-degeneracy of the noise,
it is natural to ask whether this structural assumption allows
an improvement of the strong convergence rate.
The following example shows that this is not the case:
even under uniform ellipticity, the order $1/2$ is optimal.

\begin{remark}\label{rem1.5} We show that the strong convergence order $1/2$ in Theorem~\ref{the1.4}
is optimal, even for smooth, bounded, and uniformly elliptic diffusion coefficients with vanishing drift.

\smallskip
Consider \eqref{1.1} with $d=1$, $b\equiv0$, and $\sigma(x)=2+\tanh(x)$.
Then $1<\sigma(x)<3$ and $|\sigma'(x)|\le1$ for all $x\in\mathbb{R}$.
Let $t_k=k/n$ and $\Delta W_k=W_{t_{k+1}}-W_{t_k}$.
The Euler--Maruyama scheme reads
\begin{equation}\label{1.8}
X^n_{t_{k+1}}
= X^n_{t_k} + \sigma(X^n_{t_k})\Delta W_k,
\qquad X^n_0=0 .
\end{equation}
Fix $k\in\{0,\dots,n-1\}$.
Using the It\^{o}--Taylor expansion, we obtain
\begin{equation}\label{1.9}
X_{t_{k+1}}
= X_{t_k}
+ \sigma(X_{t_k})\Delta W_k
+ \frac12\sigma(X_{t_k})\sigma'(X_{t_k})
\big((\Delta W_k)^2-\tfrac{1}{n}\big)+R_k,
\end{equation}
where $R_k$ is $\mathcal{F}_{t_{k+1}}$-measurable and satisfies
\begin{equation*}
\mathbb{E}[R_k|\mathcal{F}_{t_k}]=0,
\qquad
\mathbb{E}\!\left[|R_k|^2|\mathcal{F}_{t_k}\right]
\le Cn^{-3}.
\end{equation*}
Define the grid error $e_k=X_{t_k}-X^n_{t_k}$.
Subtracting \eqref{1.8} from \eqref{1.9} yields
\begin{equation}\label{1.10}
e_{k+1}
= e_k
+\Xi_k + B_k+R_k,
\end{equation}
where
\[
\Xi_k:=(\sigma(X_{t_k})-\sigma(X^n_{t_k}))\Delta W_k, \quad B_k=\frac12\sigma(X_{t_k})\sigma'(X_{t_k})
\big((\Delta W_k)^2-\tfrac{1}{n}\big).
\]

Observe that $(\Delta W_k)^2-\tfrac1n$ are independent, centered random variables,
and $B_k$ is $\mathcal F_{t_{k+1}}$-measurable with
$\mathbb E[B_k|\mathcal F_{t_k}]=0$.
Hence $(B_k)_{k\ge0}$ is an $L^2$-orthogonal family and
\[
\mathbb E\bigg(\sum_{k=0}^{n-1}B_k\bigg)^2
=\sum_{k=0}^{n-1}\mathbb E[B_k^2].
\]
Using $\mathbb E[((\Delta W_k)^2-\tfrac1n)^2]=2n^{-2}$, we compute
\begin{equation}\label{1.11}
\mathbb{E}\bigg(\sum_{k=0}^{n-1}B_k\bigg)^2
= \frac{1}{2n^2}\sum_{k=0}^{n-1}
\mathbb{E}\!\left[(\sigma(X_{t_k})\sigma'(X_{t_k}))^2\right]
\ge C_0n^{-1},
\end{equation}
where the last inequality follows from the fact that
$\sigma(x)\sigma'(x)\not\equiv0$ and the uniform ellipticity of $\sigma$,
which ensure that
$\mathbb E[(\sigma(X_t)\sigma'(X_t))^2]$ remains uniformly positive
over a nontrivial time interval.

\smallskip
For $\Xi_k$ and $R_k$,  we compute that
\begin{equation}\label{1.12}
\mathbb{E}\bigg(\sum_{k=0}^{n-1}\Xi_k\bigg)^2
\le C_2n^{-1}, \quad \mathbb{E}\bigg(\sum_{k=0}^{n-1}R_k\bigg)^2
\le C_1n^{-2}.
\end{equation}

Combining \eqref{1.10}--\eqref{1.12}, we obtain
\begin{equation*}
\begin{split}
\|e_{n+1}\|_{L^2(\Omega)}&=\Big\|\sum_{k=0}^{n-1}\Xi_k+\sum_{k=0}^{n-1}B_k+\sum_{k=0}^{n-1}R_k\Big\|_{L^2(\Omega)}
\geq \Big\|\sum_{k=0}^{n-1}\Xi_k+\sum_{k=0}^{n-1}B_k\Big\|_{L^2(\Omega)}-C_1n^{-1}\notag \\ &= \bigg\{\Big\|\sum_{k=0}^{n-1}\Xi_k\Big\|_{L^2(\Omega)}^2+\Big\|\sum_{k=0}^{n-1}B_k
\Big\|_{L^2(\Omega)}^2\bigg\}^{\frac{1}{2}}-C_1n^{-1}\geq C_3n^{-\frac{1}{2}}.
\end{split}
\end{equation*}
Consequently,
\[
\|X_1-X_1^n\|_{L^2(\Omega)}\geq C_3 n^{-\frac12}.
\]
\end{remark}

\section{Analytic and stochastic preliminaries}\label{sec2}\setcounter{equation}{0}
In this section we collect several analytic and stochastic tools that will be
used in the proofs of the main results.
These include properties of slowly varying functions,
regularity results for Kolmogorov equations with Lebesgue--Dini coefficients,
as well as two auxiliary lemmas from the theory of stochastic sewing and stochastic analysis.
We begin with a structural representation of slowly varying functions,
which will be used repeatedly in the sequel.
\begin{lemma}\label{lem2.1}(\cite[Lemma~2.8]{WWLD})
Let $\rho:[0,1]\to\mathbb{R}_+$ be nondecreasing and slowly varying at zero, with
$\rho(r)\to0$ as $r\to0$.
Then there exists $r_0\in(0,1]$ such that
\begin{equation}\label{2.1}
\rho(r)
=\exp\Big\{c(r)-\int_r^{r_0}\frac{\zeta(\tau)}{\tau}d\tau\Big\},
\qquad 0<r\le r_0,
\end{equation}
where $c$ is a continuous function and $\zeta$ is a nonnegative continuous
function on $(0,1]$ satisfying
\[
\lim_{r\to0}c(r)=c_0\in\mathbb{R},\qquad
\lim_{r\to0}\zeta(r)=0,\qquad
\lim_{r\to0}\int_r^{r_0}\frac{\zeta(\tau)}{\tau}d\tau=\infty.
\]
\end{lemma}

Next, we recall the existence and regularity properties of the fundamental
solution associated with uniformly elliptic Kolmogorov operators.
\begin{lemma}\label{lem2.2}
Assume that \eqref{1.5} and \eqref{1.6} hold.
Then the operator
\[
\partial_t-\tfrac12\sum_{i,j=1}^da_{i,j}(t,x)\partial^2_{x_i,x_j}
=:L_t
\]
admits a fundamental solution, denoted by $\cK(r,t,x,y)$.
Moreover, for $t>r$, $\cK$ is twice continuously differentiable with respect to $x$,
and it satisfies Aronson-type estimates (see \cite[Theorem~2.3]{CHXZ}):
\begin{equation}\label{2.2}
\left\{
\begin{aligned}
&\frac{C_1}{(t-r)^{\frac d2}}
\exp\!\Big(-\tfrac{C_2|x-y|^2}{t-r}\Big)
\le \cK(r,t,x,y)
\le \frac{C_3}{(t-r)^{\frac d2}}
\exp\!\Big(-\tfrac{C_4|x-y|^2}{t-r}\Big),\\[0.2em]
&\big|\nabla_x^k \cK(r,t,x,y)\big|
\le \frac{C_5}{(t-r)^{\frac{d+m}{2}}}
\exp\!\Big(-\tfrac{C_6|x-y|^2}{t-r}\Big),
\quad k=1,2,
\end{aligned}
\right.
\end{equation}
where $C_i>0$ ($i=1,\ldots,6$) are constants independent of $(r,t,x,y)$.
\end{lemma}

Let $\lambda>0$ and consider the Kolmogorov equation
\begin{equation}\label{2.3}
\left\{
\begin{aligned}
\partial_t u(t,x)
&=\tfrac12\sum_{i,j=1}^da_{i,j}(t,x)\partial^2_{x_i,x_j}u(t,x)
+g(t,x)\cdot\nabla u(t,x) \\
&\quad -\lambda u(t,x)+f(t,x),
\qquad (t,x)\in(0,1]\times\mathbb{R}^d,\\
u(0,x)&=0,\qquad x\in\mathbb{R}^d .
\end{aligned}
\right.
\end{equation}

A function $u$ is called a \emph{strong solution} of \eqref{2.3} if $u, \partial_tu, \partial^2_{x_i,x_j}u\in L^1\big([0,1];L^\infty_{\mathrm{loc}}(\mathbb{R}^d)\big),
\ 1\le i,j\le d$,
and \eqref{2.3} holds almost everywhere.

\begin{theorem}\label{the2.3}
Let $\lambda>0$ be sufficiently large.
Assume that $a(t,x)=(a_{i,j}(t,x))_{d\times d}$ satisfies \eqref{1.5}--\eqref{1.6},
and that $\rho:[0,1]\to\mathbb{R}_+$ is slowly varying at zero with
$\rho^{1/2}$ being a Dini function.
Suppose $f\in L^2\big([0,1];\mathcal{D}_b^\rho(\mathbb{R}^d)\big)$ and $g\in L^2\big([0,1];\mathcal{D}_b^\rho(\mathbb{R}^d;\mathbb{R}^d)\big)$.
Define
\[
\mathcal{X}
:=\Big\{
v\in L^\infty\big([0,1];W^{1,\infty}(\mathbb{R}^d)\big):
\partial_tv,\ \nabla^2v\in L^2\big([0,1];L^\infty(\mathbb{R}^d)\big)
\Big\}.
\]
Then \eqref{2.3} admits a unique strong solution $u\in\mathcal{X}$.
Moreover,
\begin{equation}\label{2.4}
\sup_{0\le t\le1}\|\nabla u(t)\|_\infty\le \tfrac12 .
\end{equation}
\end{theorem}

\begin{proof}
Let $\varrho\in \mathcal{C}_0^\infty(\mathbb{R}^d)$ be a standard spatial mollifier such that
$\mathrm{supp}(\varrho)\subset B_1$ and $\int_{\mathbb{R}^d}\varrho(x)dx=1$.
For $m\in\mathbb{N}$, set $\varrho_m(x)=m^d\varrho(mx)$ and define
\[
f^m(t,x):=(f(t,\cdot)\ast \varrho_m)(x),\qquad
g^m(t,x):=(g(t,\cdot)\ast \varrho_m)(x).
\]
Then $f^m\in L^2([0,1];\mathcal{C}_b^\alpha(\mathbb{R}^d))$ and
$g^m\in L^2([0,1];\mathcal{C}_b^\alpha(\mathbb{R}^d;\mathbb{R}^d))$ for every $\alpha\in(0,1)$, and
\begin{equation}\label{2.5}
\lim_{m\to\infty} |f^m(t,x)-f(t,x)|=0,\qquad
\lim_{m\to\infty} |g^m(t,x)-g(t,x)|=0,
\quad \forall\ (t,x)\in[0,1]\times\mathbb{R}^d.
\end{equation}
Moreover,
\begin{equation}\label{2.6}
\|f^m(r)\|_\rho\le \|f(r)\|_\rho,\qquad
\|g^m(r)\|_\rho\le \|g(r)\|_\rho.
\end{equation}

Consider the Kolmogorov equation
\begin{equation}\label{2.7}
\left\{
\begin{aligned}
\partial_t u^m(t,x)
&= \tfrac{1}{2}\sum_{i,j=1}^d a_{i,j}(t,x)\partial^2_{x_i,x_j} u^m(t,x)
+ g^m(t,x)\cdot \nabla u^m(t,x) \\
&\quad - \lambda u^m(t,x) + f^m(t,x),
\qquad (t,x)\in (0,1]\times \mathbb{R}^d, \\
u^m(0,x)&=0,\qquad x\in \mathbb{R}^d .
\end{aligned}
\right.
\end{equation}
By \cite[Theorem~2.1]{TDW}, there exists a unique
\[
u^m\in L^2\big([0,1];\mathcal{C}_b^{2+\alpha}(\mathbb{R}^d)\big)\cap
W^{1,2}\big([0,1];\mathcal{C}_b^\alpha(\mathbb{R}^d)\big)
\]
solving \eqref{2.7}. Moreover, $u^m$ has the integral representation
\begin{equation}\label{2.8}
u^m(t,x)
=\int_0^t\int_{\mathbb{R}^d} e^{-\lambda(t-r)}\cK(r,t,x,y)
\big[g^m(r,y)\cdot \nabla u^m(r,y)+f^m(r,y)\big]dydr,
\end{equation}
where $\cK(r,t,x,y)$ satisfies \eqref{2.2}.

\smallskip
Fix $x_0\in\mathbb{R}^d$ and consider
\begin{equation}\label{2.9}
\dot x_t=-g^m(t,x_0+x_t),\qquad x_t\big|_{t=0}=0.
\end{equation}
This ODE admits a global absolutely continuous solution and satisfies the uniform increment estimate:
for $0\le r\le t\le 1$,
\begin{equation}\label{2.10}
|x_t-x_r|
\le \int_r^t \|g^m(\tau)\|_0d\tau
\le (t-r)^{\frac12}\|g\|_{2,0},
\end{equation}
which is independent of $x_0$ and $m$.

Set $\varphi_t^{x_0}=x_0+x_t$ and define
\[
\hat u^m(t,x):=u^m(t,x+\varphi_t^{x_0}),\quad
\hat a_{i,j}(t,x):=a_{i,j}(t,x+\varphi_t^{x_0}),
\]
\[
\hat f^m(t,x):=f^m(t,x+\varphi_t^{x_0}),\quad
\hat g^m(t,x):=g^m(t,x+\varphi_t^{x_0}),\quad
\tilde g^m(t,x):=\hat g^m(t,x)-\hat g^m(t,0).
\]
Then $\hat u^m$ solves
\begin{equation}\label{2.11}
\left\{
\begin{aligned}
\partial_t \hat u^m(t,x)
&= \tfrac{1}{2}\sum_{i,j=1}^d\hat a_{i,j}(t,x)\partial^2_{x_i,x_j}\hat u^m(t,x)
+ \tilde g^m(t,x)\cdot \nabla \hat u^m(t,x) \\
&\quad - \lambda \hat u^m(t,x) + \hat f^m(t,x),
\qquad (t,x)\in (0,1]\times \mathbb{R}^d, \\
\hat u^m(0,x)&=0,\qquad x\in \mathbb{R}^d .
\end{aligned}
\right.
\end{equation}
Since $\hat a_{i,j}$ also satisfies \eqref{1.5} and \eqref{1.6}, the operator
\[
\partial_t-\tfrac12\sum_{i,j=1}^d\hat a_{i,j}(t,x)\partial^2_{x_i,x_j}=: \hat L_t(x)
\]
has a fundamental solution, denoted by $\hat \cK_{\varphi^{x_0}}(r,t,x,y)$, and the unique strong solution of \eqref{2.11} can be represented as
\begin{equation}\label{2.12}
\hat u^m(t,x)
=\int_0^t e^{-\lambda(t-r)}\int_{\mathbb{R}^d}\hat \cK_{\varphi^{x_0}}(r,t,x,y)
\big[\tilde g^m(r,y)\cdot \nabla\hat u^m(r,y)+\hat f^m(r,y)\big]dydr.
\end{equation}
Since $\hat L_t$ is a spatial translate of $L_t$, one has
\[
\hat \cK_{\varphi^{x_0}}(r,t,x,y)
=\cK\bigl(r,t,x+\varphi_t^{x_0},y+\varphi_r^{x_0}\bigr).
\]

Hence, by \eqref{2.2},
\begin{equation}\label{2.13}
\begin{aligned}
|\nabla \hat u^m(t,0)|
&\le \int_0^t e^{-\lambda(t-r)}\int_{\mathbb{R}^d}
\big|\nabla_x \hat \cK_{\varphi^{x_0}}(r,t,0,y)\big|
\big(|\tilde g^m(r,y)\cdot \nabla\hat u^m(r,y)|
\\ & \qquad +|\hat f^m(r,y)-\hat f^m(r,0)|\big)dydr \\
&\le C\int_0^t e^{-\lambda(t-r)}\int_{\mathbb{R}^d}
(t-r)^{-\frac{d+1}{2}}
\exp\!\Big(-\frac{C_6|\varphi^{x_0}_t-\varphi^{x_0}_r-y|^2}{t-r}\Big) \\
&\quad \times \big[|\nabla\hat u^m(r,y)|\|g(r)\|_\rho+\|f(r)\|_\rho\big]
\big[1_{|y|\le 1}\rho(|y|)+1_{|y|>1}\big]dydr \\
&\le C\int_0^t e^{-\lambda(t-r)}\int_{\mathbb{R}^d}
(t-r)^{-\frac{d+1}{2}}
\exp\!\Big(-\frac{C_6|y|^2}{2(t-r)}\Big) \\
&\quad \times \big[|\nabla\hat u^m(r,y)|\|g(r)\|_\rho+\|f(r)\|_\rho\big]
\big[1_{|y|\le 1}\rho(|y|)+1_{|y|>1}\big]dydr,
\end{aligned}
\end{equation}
where we used \eqref{2.10} to absorb the spatial shift into the Gaussian kernel.

Therefore,
\begin{equation}\label{2.14}
\begin{aligned}
|\nabla \hat u^m(t,0)|
&\le C\int_0^t e^{-\lambda(t-r)}
\big[\|\nabla u^m(r)\|_0\|g(r)\|_\rho+\|f(r)\|_\rho\big] \\
&\quad\times\bigg(\int_0^1 e^{-\frac{C_6\tau^2}{2(t-r)}}(t-r)^{-\frac{d+1}{2}}\tau^{d-1}\rho(\tau)d\tau+1\bigg)dr \\
&\le C\int_0^t e^{-\lambda(t-r)}
\big[\|\nabla u^m(r)\|_0\|g(r)\|_\rho+\|f(r)\|_\rho\big] \\
&\quad\times\bigg[\int_0^1 e^{-\frac{C_6\tau^2}{4(t-r)}}
\Big(\frac{\rho(\tau)}{\rho(\sqrt{t-r})}\Big)^{\frac12}
\frac{\rho(\sqrt{t-r})^{\frac12}}{(t-r)^{\frac12}}
\frac{\rho(\tau)^{\frac12}}{\tau}d\tau+1\bigg]dr,
\end{aligned}
\end{equation}
where we used
\[
e^{-\frac{C_6\tau^2}{2(t-r)}}(t-r)^{-\frac{d+1}{2}}\tau^{d-1}\rho(\tau)
\le
C e^{-\frac{C_6\tau^2}{4(t-r)}}
\Big(\frac{\rho(\tau)}{\rho(\sqrt{t-r})}\Big)^{\frac12}
\frac{\rho(\sqrt{t-r})^{\frac12}}{(t-r)^{\frac12}}
\frac{\rho(\tau)^{\frac12}}{\tau}.
\]

On the other hand,
\begin{equation}\label{2.15}
\begin{aligned}
\sup_{\tau\in[0,1],r\in(0,1]}
\bigg[e^{-\frac{C_6\tau^2}{4r}}
\Big(\frac{\rho(\tau)}{\rho(\sqrt{r})}\Big)^{\frac12}\bigg]
&\le 1+\sup_{\mu\ge 1}\sup_{r\in(0,1/\mu]}
\bigg[e^{-\frac{C_6\mu}{4}}
\Big(\frac{\rho(\sqrt{\mu r})}{\rho(\sqrt{r})}\Big)^{\frac12}\bigg].
\end{aligned}
\end{equation}
Choosing $r_0=1$ in Lemma~\ref{lem2.1} and using \eqref{2.1}, we obtain
\begin{equation}\label{2.16}
\begin{aligned}
\sup_{r\in(0,1/\mu]}\frac{\rho(\sqrt{\mu r})}{\rho(\sqrt{r})}
&=\sup_{r\in(0,1/\mu]}
\exp\bigg\{c(\sqrt{\mu r})-c(\sqrt{r})
+\int_{\sqrt{r}}^{\sqrt{\mu r}}\frac{\zeta(\tau)}{\tau}d\tau\bigg\} \\
&\le \exp\Big\{2\sup_{0\le \tau\le 1}|c(\tau)|
+\sup_{0\le \tau\le 1}\zeta(\tau)\log(\sqrt{\mu})\Big\}
\le C\mu^{\frac12\sup_{0\le \tau\le 1}\zeta(\tau)}.
\end{aligned}
\end{equation}
Combining \eqref{2.14}--\eqref{2.16} yields
\begin{equation}\label{2.17}
\begin{aligned}
|\nabla \hat u^m(t,0)|
&\le C\int_0^t e^{-\lambda(t-r)}
\big[\|\nabla u^m(r)\|_0\|g(r)\|_\rho+\|f(r)\|_\rho\big]\bigg[
\frac{\rho(\sqrt{t-r})^{\frac12}}{(t-r)^{\frac12}}
\int_0^1\frac{\rho(\tau)^{\frac12}}{\tau}d\tau+1\bigg]dr \\
&\le C\int_0^t e^{-\lambda(t-r)}
\big[\|\nabla u^m(r)\|_0\|g(r)\|_\rho+\|f(r)\|_\rho\big]
\bigg[\frac{\rho(\sqrt{t-r})^{\frac12}}{(t-r)^{\frac12}}+1\bigg]dr,
\end{aligned}
\end{equation}
where in the second inequality we used the assumption that $\rho^{1/2}$ is a Dini function.

Since $x_0\in\mathbb{R}^d$ is arbitrary, H\"older's inequality and \eqref{2.17} yield
\begin{equation}\label{2.18}
\sup_{0\le t\le 1}\|\nabla u^m(t)\|_0
\le
C\big[\|g\|_{2,\rho}\sup_{0\le r\le 1}\|\nabla u^m(r)\|_0+\|f\|_{2,\rho}\big]
\bigg(\int_0^1 e^{-2\lambda r^2}\Big[r+\frac{\rho(r)}{r}\Big]dr\bigg)^{\frac12},
\end{equation}
where $\|\nabla u^m(t)\|_0:=\sup_{x\in \mathbb{R}^d}|\nabla u^m(t,x)|$.

Since $\rho(r)/r\in L^1([0,1])$, the integral in parentheses tends to $0$ as $\lambda\to\infty$
by dominated convergence. Since $\lambda$ is large enough, we have
\[
C\big[\|g\|_{2,\rho}+\|f\|_{2,\rho}\big]
\bigg(\int_0^1 e^{-2\lambda r^2}\Big[r+\frac{\rho(r)}{r}\Big]dr\bigg)^{\frac12}
\le \frac13.
\]
Then \eqref{2.18} yields
\begin{equation}\label{2.19}
\sup_{0\le t\le 1}\|\nabla u^m(t)\|_0
\le \frac{3C}{2}\|f\|_{2,\rho}
\bigg(\int_0^{1} e^{-2\lambda r^2}\Big[r+\frac{\rho(r)}{r}\Big]dr\bigg)^{\frac12}
\le \frac12.
\end{equation}

This estimate, together with \eqref{2.8} and \eqref{2.2}, also implies
\begin{equation}\label{2.20}
\sup_{0\le t\le 1}\|u^m(t)\|_0
\le \int_0^1\big[\|g(r)\|_0\sup_{0\le \tau\le 1}\|\nabla u^m(\tau)\|_0+\|f(r)\|_0\big]dr
\le C\|f\|_{2,\rho}.
\end{equation}

For the second-order derivatives of $u^m$, an analogue of \eqref{2.17} gives
\[
|\nabla^2\hat u^m(t,0)|
\le C\int_0^t e^{-\lambda(t-r)}
\big[\|g(r)\|_\rho+\|f(r)\|_\rho\big](t-r)^{-\frac12}
\bigg[\frac{\rho(\sqrt{t-r})^{\frac12}}{(t-r)^{\frac12}}+1\bigg]dr,
\]
which, by Young's inequality, implies
\begin{equation}\label{2.21}
\|\nabla^2 u^m\|_{2,0}
\le C\big[\|g\|_{2,\rho}+\|f\|_{2,\rho}\big]
\int_0^1 e^{-\lambda r}\Big[r^{-\frac12}+\frac{\rho(\sqrt{r})^{\frac12}}{r}\Big]dr
\le C\big[\|g\|_{2,\rho}+\|f\|_{2,\rho}\big].
\end{equation}

Combining \eqref{2.7}, \eqref{2.19}, \eqref{2.20} and \eqref{2.21}, we obtain
\begin{equation}\label{2.22}
\|\partial_t u^m\|_{2,0}
\le C\big[\|g\|_{2,\rho}+\|f\|_{2,\rho}\big].
\end{equation}
Using \eqref{2.19}--\eqref{2.22}, \eqref{2.9} and standard compactness arguments,
there exist a subsequence (not relabelled) and a measurable $u\in\mathcal{X}$
such that $u^m\to u$ a.e.\ on $[0,1]\times\mathbb{R}^d$ as $m\to\infty$.
In particular, $u$ solves \eqref{2.3} and satisfies \eqref{2.4}.

\smallskip
Finally, we prove uniqueness. By linearity, it suffices to consider the homogeneous problem,
that is, with $f\equiv 0$. Repeating the above estimates for $u$ (in place of $u^m$) yields (see \eqref{2.20}), in particular, $\|u\|_{\infty,0}=0$, and the uniqueness follows.
\end{proof}

We now recall two auxiliary lemmas from stochastic analysis,
which will be used in the proof of the strong error estimates for numerical
approximations.
\begin{lemma}\label{lem2.4}(\cite[Lemma 2.2]{LL})
Let $\varepsilon>0$ and let $v,S,T,C_1,C_2,C_3,\Gamma_1,\Gamma_2\ge 0$ be fixed constants such that
$0\le v<S<T$.
Let $w$ be a continuous deterministic control on $\Delta([S,T])$.
Let $J$ be an $L^p$-integrable adapted process indexed by $\Delta([S,T])$ such that,
for every $(s,u,t)\in\Delta_2([S,T])$,
\begin{equation*}
\begin{aligned}
&\qquad \quad\|J_{s,t}\|_{L^p(\Omega|\mathcal{F}_v)}
\le C_2 w(s,t)^{\frac12+\varepsilon}, \quad \|\mathbb{E}_s J_{s,t}\|_{L^p(\Omega|\mathcal{F}_v)}
\le C_1 w(s,t)^{1+\varepsilon}, \\
& \|\delta J_{s,u,t}\|_{L^p(\Omega|\mathcal{F}_v)}
\le \Gamma_2 w(s,t)^{\frac12}
+ C_3\Gamma_2 w(s,t)^{\frac12+\varepsilon}, \quad
\|\mathbb{E}_s \delta J_{s,u,t}\|_{L^p(\Omega|\mathcal{F}_v)}
\le \Gamma_1 w(s,t)^{1+\varepsilon}.
\end{aligned}
\end{equation*}
Then there exists a constant $N=N(\varepsilon,p)$, in particular independent of
$\Gamma_1,\Gamma_2,C_1,C_2,C_3,S,T,v$ and $w$, such that for every $(s,t)\in\Delta([S,T])$,
\begin{equation*}
\begin{aligned}
\|J_{s,t}\|_{L^p(\Omega|\mathcal{F}_v)}
\le {}&
N\Gamma_2\Big[(1+|\log\Gamma_2|)w(s,t)^{\frac12}
+ C_1 w(s,t)^{1+\varepsilon}
+ (C_2+C_3) w(s,t)^{\frac12+\varepsilon}\Big] + N\Gamma_1 w(s,t)^{1+\varepsilon}.
\end{aligned}
\end{equation*}
\end{lemma}

\begin{lemma}\label{lem2.5}(\cite[Lemma~3.1]{LL})
Let $A=(A_t)_{t\in[0,1]}$ be a continuous adapted stochastic process, and let
$p,N\in(0,\infty)$ be fixed constants. Assume that $A_0=0$ and
\begin{equation*}
\sup_{0\le s\le t\le 1}\|\delta A_{s,t}\|_{L^p(\Omega|\mathcal{F}_s)}\le N.
\end{equation*}
Then for every $\bar p\in(0,p)$, there exists a constant $C(\bar p,p)$ such that
\begin{equation*}
\Big\|\sup_{0\le t\le 1}|A_t|\Big\|_{L^{\bar p}(\Omega)}
\le C(\bar p,p)N.
\end{equation*}
\end{lemma}

\section{Proof of Theorem \ref{the1.2}}\label{sec3}\setcounter{equation}{0}
Fix $\lambda>0$ large enough. Consider the backward Cauchy problem
\begin{equation}\label{3.1}
\left\{
\begin{aligned}
\partial_t U(t,x)
&+\tfrac12\sum_{i,j=1}^d a_{i,j}(t,x)\partial^2_{x_i,x_j}U(t,x)
+b(t,x)\cdot\nabla U(t,x)
\\
&=\lambda U(t,x)-b(t,x),
\qquad (t,x)\in[0,1)\times\mathbb{R}^d,\\
U(1,x)&=0,\qquad x\in\mathbb{R}^d.
\end{aligned}
\right.
\end{equation}
By Theorem~\ref{the2.3}, there exists a unique solution $U\in\mathcal{X}^d$ of \eqref{3.1}, where
\begin{equation*}
\mathcal{X}^d
:=\Big\{V\in L^\infty([0,1];W^{1,\infty}(\mathbb{R}^d;\mathbb{R}^d))
:\ |\partial_tV|, \ |\nabla^2V| \in L^2([0,1];L^\infty(\mathbb{R}^d))\Big\}.
\end{equation*}
Moreover, \eqref{2.4} holds for $U$. Define $\Phi(t,x):=x+U(t,x)$ and $\Psi(t,\cdot):=\Phi(t,\cdot)^{-1}$.
Then $\Phi(t,\cdot)$ is a bi-Lipschitz homeomorphism uniformly in $t\in[0,1]$, and
\begin{equation}\label{3.2}
\tfrac12\le \sup_{0\le t\le 1}\|\nabla \Phi(t,\cdot)\|_0\le \tfrac32,
\qquad
\tfrac23\le \sup_{0\le t\le 1}\|\nabla \Psi(t,\cdot)\|_0\le 2.
\end{equation}
For $0<\varepsilon<1$, define the time average
\begin{equation*}
U_\varepsilon(t,x):=\varepsilon^{-1}\int_t^{t+\varepsilon}U(r,x)dr
=\int_0^1 U(t+\varepsilon r,x)dr,
\end{equation*}
and set $\Phi_\varepsilon(t,x):=x+U_\varepsilon(t,x)$, with the convention $U(t,x)=0$ for $t>1$.

\smallskip
Let $X_{s,t}(x)$ be a strong solution to \eqref{1.1}. Applying It\^{o}'s formula
(cf.\ \cite[Theorem~3.7]{KR}) to $\Phi_\varepsilon(t,X_{s,t}(x))$ and then letting
$\varepsilon\downarrow0$ yields
\begin{equation}\label{3.3}
\Phi(t,X_{s,t}(x))
=\Phi(s,x)
+\lambda\int_s^t U\bigl(r,X_{s,r}(x)\bigr)dr
+\int_s^t \bigl[I+\nabla U\bigl(r,X_{s,r}(x)\bigr)\bigr]\sigma\bigl(r,X_{s,r}(x)\bigr)dW_r .
\end{equation}

Denote $Y_{s,t}(y):=\Phi\bigl(t,X_{s,t}(x)\bigr)$ and $y:=\Phi(s,x)$. Then $Y_{s,t}(y)$ solves the SDE
\begin{equation}\label{3.4}
dY_{s,t}(y)=\tilde b\bigl(t,Y_{s,t}(y)\bigr)dt
+\tilde\sigma\bigl(t,Y_{s,t}(y)\bigr)dW_t,
\quad t\in(s,1],\quad
Y_{s,s}(y)=y,
\end{equation}
where, for $(t,z)\in[0,1]\times\mathbb{R}^d$,
\begin{equation*}
\tilde b(t,z):=\lambda U\bigl(t,\Psi(t,z)\bigr),
\qquad
\tilde\sigma(t,z):=\bigl[I+\nabla U\bigl(t,\Psi(t,z)\bigr)\bigr]
\sigma\bigl(t,\Psi(t,z)\bigr).
\end{equation*}
Conversely, if $Y_{s,t}(y)$ is a strong solution of \eqref{3.4}, then by \eqref{3.2} and It\^{o}'s formula,
$X_{s,t}(x):=\Psi\bigl(t,Y_{s,t}(y)\bigr)$
is a strong solution of \eqref{1.1}. Hence \eqref{1.1} and \eqref{3.4} are equivalent.

\smallskip
The regularity of $U$ and $\sigma$ implies that
$\tilde b\in L^\infty([0,1];\mathrm{Lip}(\mathbb{R}^d;\mathbb{R}^d))$
and
$\tilde\sigma\in L^2([0,1];W^{1,\infty}(\mathbb{R}^d;\mathbb{R}^{d\times d}))$.
Therefore, by the Cauchy--Lipschitz (Picard--Lindel\"of) theorem,
for every $0\le s\le 1$ and $y\in\mathbb{R}^d$ there exists a unique strong solution
$Y_{s,t}(y)$ to \eqref{3.4}, and the flow property holds: $Y_{s,t}(y)=Y_{r,t}\bigl(Y_{s,r}(y)\bigr)$ for every $0\le s\le r\le t\le 1$.

\smallskip
Applying It\^{o}'s formula to $|Y_{s,t}(y)|^p$ yields, for any $p\ge 2$,
\begin{equation*}
d|Y_{s,t}(y)|^p
\le C\bigl(1+|Y_{s,t}(y)|^p\bigr)dt
+p|Y_{s,t}(y)|^{p-2}\big\langle Y_{s,t}(y),\tilde\sigma(t,Y_{s,t}(y))dW_t\big\rangle.
\end{equation*}
Hence,
\begin{equation}\label{3.5}
\sup_{s\le t\le 1}\mathbb{E}|Y_{s,t}(y)|^p \le C\bigl(1+|y|^p\bigr),
\end{equation}
where we used the fact the stochastic integral is a martingale with zero expectation.

\smallskip
To verify that $\{Y_{s,t}(\cdot),t\in[s,1]\}$ is a stochastic flow of homeomorphisms,
it suffices (cf.\ \cite[Lemmas~II.2.4, II.4.1, II.4.2]{Kun84}) to show that, for all
$y,y'\in\mathbb{R}^d$ and $s<t$, $s'<t'$,
\begin{equation}\label{3.6}
\sup_{s\le t\le 1}\mathbb{E}\bigl|Y_{s,t}(y)-Y_{s,t}(y')\bigr|^{2\xi}
\le C|y-y'|^{2\xi},
\qquad \text{for all }\xi<0,
\end{equation}
and
\begin{equation}\label{3.7}
\mathbb{E}\bigl|Y_{s,t}(y)-Y_{s',t'}(y')\bigr|^p
\le C\Big\{|y-y'|^p
+|s-s'|^{\tfrac{p}{2}}+|t-t'|^{\tfrac{p}{2}}\Big\},
\quad p\ge 2.
\end{equation}

We first prove \eqref{3.6}. Fix $\xi<0$ and $\varepsilon>0$, and set
$f_\varepsilon(x):=\varepsilon+|x|^2$ and $\Delta Y_{s,t}:=Y_{s,t}(y)-Y_{s,t}(y')$.
Applying It\^{o}'s formula to $f_\varepsilon(\Delta Y_{s,t})^\xi$ gives
\begin{equation}\label{3.8}
\begin{aligned}
f_\varepsilon(\Delta Y_{s,t})^\xi
\le& f_\varepsilon(y-y')^\xi
+ C|\xi|\int_s^t f_\varepsilon(\Delta Y_{s,r})^\xi dr
+ C|\xi(\xi-1)|\int_s^t \psi(r)f_\varepsilon(\Delta Y_{s,r})^\xi dr \\
&\quad
+2\xi\int_s^t f_\varepsilon(\Delta Y_{s,r})^{\xi-1}
\Big\langle \Delta Y_{s,r},
\big(\tilde\sigma(r,Y_{s,r}(y))-\tilde\sigma(r,Y_{s,r}(y'))\big)dW_r\Big\rangle,
\end{aligned}
\end{equation}
where $\psi(r):=\|\nabla^2 U(r)\|_\infty^2\in L^1([0,1])$ (since $\nabla U\in L^2([0,1];W^{1,\infty}(\mathbb{R}^d;\mathbb{R}^{d\times d}))$).
Taking expectations, using that the stochastic integral has mean zero, and applying Gr\"onwall's inequality, we obtain
\begin{equation*}
\sup_{s\le t\le 1}\mathbb{E}f_\varepsilon(\Delta Y_{s,t})^\xi
\le Cf_\varepsilon(y-y')^\xi.
\end{equation*}
The desired result \eqref{3.6} follows by letting $\varepsilon\downarrow0$.

\smallskip
We next prove \eqref{3.7}. Without loss of generality, assume $s<s'<t<t'$.
Then
\begin{equation}\label{3.9}
|Y_{s,t}(y)-Y_{s',t'}(y')|^p
\le 3^{p-1}\Big(
|Y_{s,t}(y)-Y_{s,t}(y')|^p
+|Y_{s,t}(y')-Y_{s',t}(y')|^p
+|Y_{s',t}(y')-Y_{s',t'}(y')|^p
\Big).
\end{equation}
The first term is controlled by the Lipschitz dependence on initial data:
\begin{equation}\label{3.10}
\sup_{s\le t\le 1}\mathbb{E}|Y_{s,t}(y)-Y_{s,t}(y')|^p \le C|y-y'|^p.
\end{equation}
For the second and third terms, standard moment estimates for Lipschitz SDEs give
\begin{equation}\label{3.11}
\mathbb{E}|Y_{s,t}(y')-Y_{s',t}(y')|^p
\le C|s-s'|^{\frac p2},\quad
\mathbb{E}|Y_{s',t}(y')-Y_{s',t'}(y')|^p
\le C|t-t'|^{\frac p2}.
\end{equation}
Combining \eqref{3.9}--\eqref{3.11} yields \eqref{3.7}.

\smallskip
Therefore, $Y_{s,t}(\cdot)$ is a homeomorphism for every $0\le s\le t\le 1$.
Finally, since
\[
Y_{s,t}(Y_{s,t}^{-1}(y))=y
\]
and
\begin{equation*}
Y_{s,t}(Y_{s,t}^{-1}(y))
=Y_{s,t}^{-1}(y)
+\int_s^t \tilde b\bigl(r,Y_{s,r}(Y_{s,t}^{-1}(y))\bigr)dr
+\int_s^t \tilde\sigma\bigl(r,Y_{s,r}(Y_{s,t}^{-1}(y))\bigr)dW_r,
\end{equation*}
we have
\[
Y_{s,r}(Y_{s,t}^{-1}(y))=Y_{r,t}^{-1}(y)
\]
and thus
\begin{equation}\label{3.12}
Y_{s,t}^{-1}(y)
= y
-\int_s^t \tilde b\bigl(r,Y_{r,t}^{-1}(y)\bigr)dr
-\int_s^t \tilde\sigma\bigl(r,Y_{r,t}^{-1}(y)\bigr)dW_r.
\end{equation}
The same estimate as in \eqref{3.7} (applied to the backward representation \eqref{3.12})
yields the continuity of $Y_{s,t}^{-1}(y)$ in $(s,t,y)$, $\mathbb{P}$-a.s.
Hence $\{Y_{s,t}(\cdot),t\in[s,1]\}$ forms a stochastic flow of homeomorphisms for \eqref{3.4}.

\smallskip
We now return to the original process $X_{s,t}(x)$. Recall that
\[
X_{s,t}(x)
:=\Psi\bigl(t,Y_{s,t}(\Phi(s,x))\bigr), \quad 0\le s\le t\le 1,\ \ x\in\mathbb{R}^d.
\]
Since $\Phi(t,\cdot)$ and $\Psi(t,\cdot)$ are bi-Lipschitz homeomorphisms
uniformly in $t\in[0,1]$,
and since $\{Y_{s,t}(\cdot)\}_{0\le s\le t\le 1}$ is a stochastic flow of homeomorphisms,
it follows immediately that $\{X_{s,t}(\cdot)\}_{0\le s\le t\le 1}$
is also a family of homeomorphisms on $\mathbb{R}^d$ associated with \eqref{1.1}. This completes the proof of Theorem \ref{the1.2}. \qed

\section{Proof of Theorem \ref{the1.4}}\label{sec4}
\setcounter{equation}{0}
Let $\mathbb{D}_n:=\{k/n:\ k=0,1,\dots,n\}$.
For each $s\in\mathbb{D}_n$ and $x\in\mathbb{R}^d$, let $\bar X^{n}_{s,\cdot}(x)$ be the solution of the following Euler--Maruyama scheme
\begin{equation}\label{4.1}
\bar X^{n}_{s,t}(x)
=
x+\int_s^t \sigma\big(r,\bar X^{n}_{s,k_n(r)}(x)\big)dW_r,
\qquad t\in[s,1].
\end{equation}
When $s=0$, we write $\bar X^n_t(x):=\bar X^{n}_{0,t}(x)$, and if the starting point is clear we simply write $\bar X^n_t$.

\begin{lemma}\label{lem4.1} Suppose that $a:=\sigma\sigma^{\mathsf T}$ satisfies \eqref{1.5}--\eqref{1.6} with $\alpha=1$. Let $\bar X^n$ be given by \eqref{4.1}.
Assume that
\[
f\in L^2([0,1];\cC_b(\mathbb{R}^d)), \quad g\in L^\infty([0,1]\times\mathbb{R}^d)\cap L^2([0,1];W^{1,\infty}(\mathbb{R}^d)).
\]
Then, for every $t\in[0,1]$, every $p\ge 2$, and every integer $n\ge 2$, we have
\begin{equation}\label{4.2}
\Big\|
\sup_{0\le r\le t}\Big|
\int_0^r g(\tau,\bar X_\tau^n)\Big(f(\tau,\bar X_\tau^n)-f(\tau,\bar X_{k_n(\tau)}^n)\Big)d\tau
\Big|
\Big\|_{L^p(\Omega)}
\le
Cn^{-\frac12}\log(n)^{\frac32},
\end{equation}
where $C$ depends only on $p$, $d$,
$\|\nabla a\|_{\infty,\infty}, \|g\|_{\infty,\infty}, \|\nabla g\|_{2,\infty}$ and $\|f\|_{2,0}$.
\end{lemma}
\begin{proof} By Lemma~\ref{lem2.5}, it suffices to show that for every $0\le s\le t\le 1$,
\begin{equation}\label{4.3}
\bigg\|
\int_s^t g(r,\bar X_r^n)\Big(f(r,\bar X_r^n)-f(r,\bar X_{k_n(r)}^n)\Big)dr
\bigg\|_{L^p(\Omega\mid\mathcal{F}_s)}
\le
Cn^{-\frac12}\log(n)^{\frac32}.
\end{equation}
We assume without loss of generality that $n\ge4$.
If $t\le k_n(s)+4/n$, then \eqref{4.3} follows from the Cauchy--Schwarz:
\[
\bigg\|\int_s^t g(r,\bar X_r^n)\Big(f(r,\bar X_r^n)-f(r,\bar X_{k_n(r)}^n)\Big)dr\bigg\|_{L^p(\Omega\mid\mathcal{F}_s)}
\le
\|g\|_{\infty,\infty}\int_s^t 2\|f(r)\|_{0}dr
\le
C\|f\|_{2,0,[s,t]}n^{-\frac12}.
\]
Hence it remains to treat the case $t>k_n(s)+4/n$.

For $(s,t)\in[0,1]_\le^2$, define
\[
A_{s,t}
:=
\mathbb{E}_{s}
\int_s^t g(r,\bar X_r^n)\Big(f(r,\bar X_r^n)-f(r,\bar X_{k_n(r)}^n)\Big)dr.
\]
Let $0\le v<s$.
When $t\le k_n(s)+4/n$, we have
\begin{equation}\label{4.4}
\|A_{s,t}\|_{L^p(\Omega\mid\mathcal{F}_v)}
\le
4\|g\|_{\infty,\infty}\|f\|_{2,0,[s,t]}n^{-\frac12}.
\end{equation}
When $t>k_n(s)+4/n$, set $\bar s:=k_n(s)+1/n$ and write
\begin{equation}\label{4.5}
\begin{aligned}
\|A_{s,t}\|_{L^p(\Omega\mid\mathcal{F}_v)}
&\le
4\|g\|_{\infty,\infty}\|f\|_{2,0,[s,t]}n^{-\frac12}
\\
&\quad+
\int_{k_n(s)+\frac{4}{n}}^t
\bigg\|
\mathbb{E}_{\bar s}\Big[g(r,\bar X_r^n)\big(f(r,\bar X_r^n)-f(r,\bar X_{k_n(r)}^n)\big)\Big]
\bigg\|_{L^p(\Omega\mid\mathcal{F}_v)}dr,
\end{aligned}
\end{equation}
where we used $\mathbb{E}\big(|\mathbb{E}_sY|^p|\mathcal{F}_v\big)
\le \mathbb{E}\big(|\mathbb{E}_{\bar s}Y|^p|\mathcal{F}_v\big)$ for $v<s<\bar{s}$.

\smallskip
For $h\in L^2([0,1];\cC_b(\mathbb{R}^d))$, define the Markov operator
\begin{equation}\label{4.6}
Q_{s,t}^n h(\tau,x):=\mathbb{E}h\big(\tau,\bar X^{n}_{s,t}(x)\big), \quad \tau\in [0,1].
\end{equation}
By the Markov property, for $r>\bar s$,
\begin{equation}\label{4.7}
\begin{aligned}
&\mathbb{E}_{\bar s}\Big[g(r,\bar X_r^n)\big(f(r,\bar X_r^n)-f(r,\bar X_{k_n(r)}^n)\big)\Big]
\\
&=
\mathbb{E}_{\bar s}\big[(gf)(r,\bar X_r^n)-(gf)(r,\bar X_{k_n(r)}^n)\big]
+\mathbb{E}_{\bar s}\Big(\big[g(r,\bar X_{k_n(r)}^n)-g(r,\bar X_r^n)\big]f(r,\bar X_{k_n(r)}^n)\Big)
\\
&=
\big(Q_{\bar s,r}^n-Q_{\bar s,k_n(r)}^n\big)(gf)\big(r,\bar X_{\bar s}^n\big)
+
Q_{\bar s,k_n(r)}^n\Big(\big[g-Q^n_{k_n(r),r}g\big]f\Big)\big(r,\bar X_{\bar s}^n\big)
=:I^{1,n}_{\bar s,k_n(r),r}+I^{2,n}_{\bar s,k_n(r),r}.
\end{aligned}
\end{equation}

For $s\le t$ and $y\in\mathbb{R}^d$, we set
\begin{equation}\label{4.8}
T_{s,t}h(\tau,x):=\int_{\mathbb{R}^d} h(\tau,y)p_{\Sigma_{s,t}(y)}(y-x)dy,
\qquad
\end{equation}
where
\[
\Sigma_{s,t}(y):=\tfrac12\int_s^t a(\tau,y)d\tau,
\qquad
p_{\Sigma_{s,t}}(z):=\frac{(\det \Sigma_{s,t}^{-1})^{\frac12}}{(2\pi)^{\frac{d}{2}}}
\exp\!\Big(-\tfrac12 z^T\Sigma_{s,t}^{-1}z\Big).
\]
Then $T_{s,s}h(\tau,x)=h(\tau,x)$ and for $s<t$, $T_{s,t}h(\tau,\cdot)$ is smooth and satisfies
\begin{equation}\label{4.9}
\partial_s T_{s,t}h(\tau,x)
=-\partial^2_{x_i x_j}T_{s,t}\big(a_{ij}(s,\cdot)h(\tau,\cdot)\big)(x).
\end{equation}

Applying It\^o's formula to $s\mapsto T_{s,r}(gf)\big(r,\bar X_s^n\big)$ and using \eqref{4.9},
we obtain for $r>\bar s$,
\begin{equation}\label{4.10}
Q_{\bar s,r}^n(gf)\big(r,\bar X_{\bar s}^n\big)
=
T_{\bar s,r}(gf)\big(r,\bar X_{\bar s}^n\big)
+
\int_{\bar s}^{r}
Q_{\bar s,k_n(\tau)}^n\Big(H_{\tau,r}^n(gf)\Big)\big(r,\bar X_{\bar s}^n\big)d\tau,
\end{equation}
where
\begin{equation}\label{4.11}
H_{\tau,r}^n(gf)(r,x)
=
\int_{\mathbb{R}^d} \cK_{\tau,r}^n(x,y)(gf)(r,y)dy,
\end{equation}
and, writing $z:=y-x$,
\begin{equation}\label{4.12}
\left\{
\begin{aligned}
\cK_{\tau,r}^n(x,y)
&=
\big(a_{ij}(\tau,x)-a_{ij}(\tau,y)\big)
\Big(\big[(\Theta_{\tau,r}(x,y)z)_i(\Theta_{\tau,r}(x,y)z)_j\big]
\\[-1mm]
&\qquad\qquad\qquad -(\Theta_{\tau,r}(x,y))_{ij}\Big)
p_{\Sigma_{k_n(\tau),\tau}(x)+\Sigma_{\tau,r}(y)}(z),
\\
\Theta_{\tau,r}(x,y)
&:=
\big(\Sigma_{k_n(\tau),\tau}(x)+\Sigma_{\tau,r}(y)\big)^{-1}.
\end{aligned}
\right.
\end{equation}

By \eqref{4.7} and \eqref{4.10},
\[
\begin{aligned}
I^{1,n}_{\bar s,k_n(r),r}
&=
\big(T_{\bar s,r}-T_{\bar s,k_n(r)}\big)(gf)\big(r,\bar X_{\bar s}^n\big)
+\int_{k_n(r)}^{r} Q_{\bar s,k_n(\tau)}^n\Big(H_{\tau,r}^n(gf)\Big)\big(r,\bar X_{\bar s}^n\big)d\tau
\\
&\quad+
\int_{\bar s}^{k_n(r)}
Q_{\bar s,k_n(\tau)}^n\Big(\big(H_{\tau,r}^n-H_{\tau,k_n(r)}^n\big)(gf)\Big)\big(r,\bar X_{\bar s}^n\big)d\tau.
\end{aligned}
\]
Hence,
\begin{equation}\label{4.13}
\begin{aligned}
\|I^{1,n}_{\bar s,k_n(r),r}\|_{L^p(\Omega\mid\mathcal{F}_v)}
&\le
\big\|\big(T_{\bar s,r}-T_{\bar s,k_n(r)}\big)(gf)(r)\big\|_\infty
+\int_{k_n(r)}^{r}\big\|H_{\tau,r}^n(gf)(r)\big\|_\infty d\tau
\\
&\quad+
\int_{\bar s}^{k_n(r)}
\big\|\big(H_{\tau,r}^n-H_{\tau,k_n(r)}^n\big)(gf)(r)\big\|_\infty d\tau.
\end{aligned}
\end{equation}

We now estimate the three terms in \eqref{4.13}.
First, by \eqref{4.8},
\begin{equation}\label{4.14}
\big\|\big(T_{\bar s,r}-T_{\bar s,k_n(r)}\big)(gf)(r)\big\|_\infty
\le
\|g\|_{\infty,\infty}\|f(r)\|_0
\sup_{x\in\mathbb{R}^d}\int_{\mathbb{R}^d}
\big|p_{\Sigma_{\bar s,r}(y)}(y-x)-p_{\Sigma_{\bar s,k_n(r)}(y)}(y-x)\big|dy.
\end{equation}
To bound the kernel difference we use \cite[Proposition~2.7]{DGL}.
Since
\[
\Sigma_{\bar s,r}(y)-\Sigma_{\bar s,k_n(r)}(y)
=\tfrac12\int_{k_n(r)}^{r} a(\tau,y)d\tau,
\qquad
\|\Sigma_{\bar s,k_n(r)}(y)^{-1}\|\le C(k_n(r)-\bar s)^{-1},
\]
we obtain
\[
\big\|I-\Sigma_{\bar s,r}(y)\Sigma_{\bar s,k_n(r)}(y)^{-1}\big\|
\le C(r-k_n(r))(k_n(r)-\bar s)^{-1}.
\]
Therefore, \cite[Proposition~2.7]{DGL} yields
\[
\begin{aligned}
&\big|p_{\Sigma_{\bar s,r}(y)}(y-x)-p_{\Sigma_{\bar s,k_n(r)}(y)}(y-x)\big|
\\
&\le
C(r-k_n(r))(k_n(r)-\bar s)^{-1}
\Big(p_{\Sigma_{\bar s,r}(y)/2}(y-x)+p_{\Sigma_{\bar s,k_n(r)}(y)/2}(y-x)\Big).
\end{aligned}
\]
Since $r-k_n(r)\le n^{-1}$ and $k_n(r)-\bar s\ge 3n^{-1}$, we further have
\[
(r-k_n(r))(k_n(r)-\bar s)^{-1}
\le
(k_n(r)-\bar s)^{-\frac12}n^{-\frac12}.
\]
Inserting this into \eqref{4.14} we obtain
\begin{equation}\label{4.15}
\big\|\big(T_{\bar s,r}-T_{\bar s,k_n(r)}\big)(gf)(r)\big\|_\infty
\le
C\|f(r)\|_0\big(k_n(r)-\bar s\big)^{-\frac12}n^{-\frac12}.
\end{equation}

Next, by \eqref{4.12} and the Lipschitz regularity of $a$,
\begin{equation}\label{4.16}
\begin{aligned}
|\cK_{\tau,r}^n(x,y)|
&\le
\|\nabla a\|_{\infty,\infty}|z|
\Big(\|\Theta_{\tau,r}(x,y)\|^2|z|^2+\|\Theta_{\tau,r}(x,y)\|\Big)
p_{\Sigma_{k_n(\tau),\tau}(x)+\Sigma_{\tau,r}(y)}(z)
\\
&\le
C|z|\Big((r-k_n(\tau))^{-2}|z|^2+(r-k_n(\tau))^{-1}\Big)
p_{\Sigma_{k_n(\tau),\tau}(x)+\Sigma_{\tau,r}(y)}(z),
\end{aligned}
\end{equation}
where we used $\|\Theta_{\tau,r}(x,y)\|\le C(r-k_n(\tau))^{-1}$.
Consequently, by \eqref{4.11} and \eqref{4.16},
\begin{equation}\label{4.17}
\int_{k_n(r)}^{r}\big\|H_{\tau,r}^n(gf)(r)\big\|_\infty d\tau
\le
C\|g\|_{\infty,\infty}\|f(r)\|_0
\int_{k_n(r)}^{r}(r-k_n(\tau))^{-\frac12}d\tau
\le
C\|f(r)\|_0n^{-\frac12}.
\end{equation}
Similarly, one obtains
\begin{equation}\label{4.18}
\int_{\bar s}^{k_n(r)}
\big\|\big(H_{\tau,r}^n-H_{\tau,k_n(r)}^n\big)(gf)(r)\big\|_\infty d\tau
\le
C\|g\|_{\infty,\infty}\|\nabla a\|_{\infty,\infty}\|f(r)\|_0n^{-\frac12}.
\end{equation}
Combining \eqref{4.13}, \eqref{4.15}, \eqref{4.17} and \eqref{4.18}, we conclude that
\begin{equation}\label{4.19}
\|I^{1,n}_{\bar s,k_n(r),r}\|_{L^p(\Omega\mid\mathcal{F}_v)}
\le
C\|f(r)\|_0\big(k_n(r)-\bar s\big)^{-\frac12}n^{-\frac12}.
\end{equation}

We now control $I^{2,n}$.
By \eqref{4.6}--\eqref{4.7},
\begin{equation}\label{4.20}
\|I^{2,n}_{\bar s,k_n(r),r}\|_{L^p(\Omega\mid\mathcal{F}_v)}
\le
\|Q^n_{k_n(r),r}g(r)-g(r)\|_\infty\|f(r)\|_0.
\end{equation}
Using \eqref{4.10} with $g$ in place of $(gf)$,
\[
Q^n_{k_n(r),r}g(r,x)-g(r,x)
=
T_{k_n(r),r}g(r,x)-g(r,x)
+
\int_{k_n(r)}^{r} H_{\tau,r}^n g(r,x)d\tau
=:I^{2,n}_1(r,x)+I^{2,n}_2(r,x).
\]
Similar to \eqref{4.14}, one derives
\begin{equation*}
\|I^{2,n}_1(r)\|_\infty
\le
C\Big(\|g\|_{\infty,\infty}\|\nabla a\|_{\infty,\infty}+\|\nabla g(r)\|_\infty\Big)n^{-\frac12}
\le
C\big(1+\|\nabla g(r)\|_\infty\big)n^{-\frac12}.
\end{equation*}
Moreover, by \eqref{4.17} with $f\equiv1$,
\[
\|I^{2,n}_2(r)\|_\infty\le C\|g\|_{\infty,\infty}n^{-\frac12}.
\]
Together with \eqref{4.20} we obtain
\begin{equation}\label{4.21}
\|I^{2,n}_{\bar s,k_n(r),r}\|_{L^p(\Omega\mid\mathcal{F}_v)}
\le
C\Big(\|f(r)\|_0+\|f(r)\|_0\|\nabla g(r)\|_\infty\Big)n^{-\frac12}.
\end{equation}

Plugging \eqref{4.19} and \eqref{4.21} into \eqref{4.5} yields
\begin{equation}\label{4.22}
\begin{aligned}
\|A_{s,t}\|_{L^p(\Omega\mid\mathcal{F}_v)}
&\le
C\|f\|_{2,0,[s,t]}n^{-\frac12}
\\
&\quad+
C
\int_{k_n(s)+\frac{4}{n}}^{t}
\Big(\|f(r)\|_0\big(k_n(r)-\bar s\big)^{-\frac12}
+\|f(r)\|_0\|\nabla g(r)\|_\infty\Big)drn^{-\frac12}
\\
&\le
C\Big(
\|f\|_{2,0,[s,t]}\log(n)^{\frac12}
+\|f\|_{2,0,[s,t]}\|\nabla g\|_{2,0,[s,t]}
\Big)n^{-\frac12}.
\end{aligned}
\end{equation}

Next define the continuous process
\[
\mathcal{A}_t:=\int_0^t g(r,\bar X_r^n)\Big(f(r,\bar X_r^n)-f(r,\bar X_{k_n(r)}^n)\Big)dr,
\qquad
\delta\mathcal{A}_{s,t}:=\mathcal{A}_t-\mathcal{A}_s,
\]
and set $J_{s,t}:=\delta\mathcal{A}_{s,t}-A_{s,t}$.
Then $\mathbb{E}_s J_{s,t}=0$, and for $v<s\le t$,
\[
\|J_{s,t}\|_{L^p(\Omega\mid\mathcal{F}_v)}
\le
\|\delta\mathcal{A}_{s,t}\|_{L^p(\Omega\mid\mathcal{F}_v)}+\|A_{s,t}\|_{L^p(\Omega\mid\mathcal{F}_v)}
\le
4\|g\|_{\infty,\infty}\|f\|_{1,0,[s,t]}.
\]
Let
\[
w(s,t):=\|f\|_{1,0,[s,t]}+\|f\|_{2,0,[s,t]}^2+\|f\|_{2,0,[s,t]}\|\nabla g\|_{2,0,[s,t]},
\]
which is a continuous control on $\Delta([0,1])$.
Moreover, $\delta J_{s,u,t}=-\delta A_{s,u,t}$ and \eqref{4.22} implies
\[
\|\delta J_{s,u,t}\|_{L^p(\Omega\mid\mathcal{F}_v)}
\le
Cn^{-\frac12}\log(n)^{\frac12}w(s,t)^{\frac12}
+
Cn^{-\frac12}w(s,t).
\]
Applying Lemma~\ref{lem2.4}, we deduce that for $v<s\le t$ and $n\ge 2$
\begin{equation*}
\|J_{s,t}\|_{L^p(\Omega\mid\mathcal{F}_v)}
\le
Cn^{-\frac12}\log(n)^{\frac32}.
\end{equation*}
This, together with the definition of $J_{s,t}$, implies
\begin{equation}\label{4.23}
\|\delta\mathcal{A}_{s,t}\|_{L^p(\Omega\mid\mathcal{F}_v)}
\le
Cn^{-\frac12}\log(n)^{\frac32}.
\end{equation}
By \eqref{4.23}, for $t>\kappa_n(s)+4/n$, then
\begin{equation*}
\begin{aligned}
&\bigg\|
\int_s^t g(r,\bar X_r^n)
\big(f(r,\bar X_r^n)-f(r,\bar X_{k_n(r)}^n)\big)dr
\bigg\|_{L_p(\Omega|\mathcal F_s)}
\notag \\ &\le \bigg\|
\int_s^{\kappa_n(s)+\frac{4}{n}} g(r,\bar X_r^n)
\big(f(r,\bar X_r^n)-f(r,\bar X_{k_n(r)}^n)\big)dr
\bigg\|_{L_p(\Omega|\mathcal F_s)}\notag\\ &\quad +\bigg\|
\int_{\kappa_n(s)+\frac{4}{n}}^t g(r,\bar X_r^n)
\big(f(r,\bar X_r^n)-f(r,\bar X_{k_n(r)}^n)\big)dr
\bigg\|_{L_p(\Omega|\mathcal F_s)} \le Cn^{-\frac{1}{2}}\log(n)^{\frac32},
\end{aligned}
\end{equation*}
which gives \eqref{4.3} and hence \eqref{4.2}. The proof is complete.
\end{proof}

\smallskip
\noindent\textbf{Proof of Theorem \ref{the1.4}.}
Let $X$ and $X^n$ be defined by \eqref{1.1} and \eqref{1.4}, respectively.
Then for $t\in[0,1]$,
\begin{equation}\label{4.24}
X_t-X_t^n
=
\int_0^t\Big(b(r,X_r)-b\big(r,X^n_{k_n(r)}\big)\Big)dr
+
\int_0^t\Big(\sigma(r,X_r)-\sigma\big(r,X^n_{k_n(r)}\big)\Big)dW_r.
\end{equation}

Let $U$ be the unique strong solution to \eqref{3.1}.
Then
\begin{equation}\label{4.25}
U\in L^\infty([0,1];W^{1,\infty}(\mathbb{R}^d;\mathbb{R}^d))
\cap
L^2([0,1];W^{2,\infty}(\mathbb{R}^d;\mathbb{R}^d)),
\end{equation}
and for $\lambda$ large enough,
\begin{equation}\label{4.26}
\sup_{0\le t\le 1}\|\nabla U(t)\|_\infty\le \tfrac12.
\end{equation}

By It\^o's formula and \eqref{3.1},
\begin{equation}\label{4.27}
dU(t,X_t)
=
\lambda U(t,X_t)dt-b(t,X_t)dt+\nabla U(t,X_t)\sigma(t,X_t)dW_t,
\end{equation}
and
\begin{equation}\label{4.28}
\begin{split}
dU(t,X_t^n)
&=
\lambda U(t,X_t^n)dt-b(t,X_t^n)dt
+\nabla U(t,X_t^n)\Big(b\big(t,X^n_{k_n(t)}\big)-b(t,X_t^n)\Big)dt
\\
&\quad+\frac12\nabla^2U(t,X_t^n)\Big(a\big(t,X^n_{k_n(t)}\big)-a(t,X_t^n)\Big)dt
+\nabla U(t,X_t^n)\sigma\big(t,X^n_{k_n(t)}\big)dW_t.
\end{split}
\end{equation}
Combining \eqref{4.24}, \eqref{4.27} and \eqref{4.28} yields
\begin{equation}\label{4.29}
\begin{split}
X_t-X_t^n
&=
U(t,X_t^n)-U(t,X_t)
+\lambda\int_0^t\Big(U(r,X_r)-U(r,X_r^n)\Big)dr
\\
&\quad-\frac12\int_0^t \nabla^2U(r,X_r^n)\Big(a\big(r,X^n_{k_n(r)}\big)-a(r,X_r^n)\Big)dr
\\
&\quad+\int_0^t\big(\nabla U(r,X_r)+I\big)\Big(\sigma(r,X_r)-\sigma\big(r,X^n_{k_n(r)}\big)\Big)dW_r
\\
&\quad+\int_0^t\Big(\nabla U(r,X_r)-\nabla U(r,X_r^n)\Big)\sigma\big(r,X^n_{k_n(r)}\big)dW_r
\\
&\quad+\int_0^t \nabla U(r,X_r^n)\Big(b(r,X_r^n)-b\big(r,X^n_{k_n(r)}\big)\Big)dr.
\end{split}
\end{equation}

Fix $p\ge2$, by \eqref{4.25}, \eqref{4.26}, \eqref{4.29} and the Burkholder--Davis--Gundy inequality, we infer
\begin{equation*}
\begin{split}
&\tfrac12\Big\|\sup_{0\le r\le t}|X_r-X_r^n|\Big\|_{L^p(\Omega)}
 \\ &\le
C\int_0^t \|\nabla^2U(r)\|_\infty\|\nabla a\|_{\infty,\infty}
\|X^n_{k_n(r)}-X_r^n\|_{L^p(\Omega)}dr
\\
&\quad +C\int_0^t \|X_r-X_r^n\|_{L^p(\Omega)}dr
+C\bigg(\int_0^t \|\nabla\sigma(r)\|_\infty^2\|X_r^n-X^n_{k_n(r)}\|_{L^p(\Omega)}^2dr\bigg)^{\tfrac12}
\\
&\quad
+C\bigg(\int_0^t \big(\|\nabla\sigma(r)\|_\infty^2+\|\sigma\|_{\infty,\infty}^2
\|\nabla^2U(r)\|_\infty^2\big)\|X_r-X_r^n\|_{L^p(\Omega)}^2dr\bigg)^{\tfrac12}
\\
&\quad
+\bigg\|
\sup_{0\le r\le t}\Big|
\int_0^r \nabla U(\tau,X_\tau^n)\Big(b(\tau,X_\tau^n)-b\big(\tau,X^n_{k_n(\tau)}\big)\Big)d\tau
\Big|\bigg\|_{L^p(\Omega)}.
\end{split}
\end{equation*}

In view of \eqref{1.4},
\[
\sup_{0\le r\le 1}\|X_r^n-X^n_{k_n(r)}\|_{L^p(\Omega)}
\le
C\big(\|b\|_{2,0}+\|\sigma\|_{\infty,0}\big)n^{-\frac12}.
\]
Hence,
\begin{equation}\label{4.30}
\begin{split}
\Big\|\sup_{0\le r\le t}|X_r-X_r^n|\Big\|_{L^p(\Omega)}^2
&\le
C\int_0^t f(r)\Big\|\sup_{0\le \tau\le r}|X_\tau-X_\tau^n|\Big\|_{L^p(\Omega)}^2dr
+Cn^{-1}
\\
&\quad+
C\bigg\|
\sup_{0\le r\le t}\Big|
\int_0^r \nabla U(\tau,X_\tau^n)\Big(b(\tau,X_\tau^n)-b\big(\tau,X^n_{k_n(\tau)}\big)\Big)d\tau
\Big|
\bigg\|_{L^p(\Omega)}^2.
\end{split}
\end{equation}
where
\[
f(r)=1+\|\nabla\sigma(r)\|_\infty^2+\|\nabla^2U(r)\|_\infty^2\in L^1([0,1]).
\]

Let $Z$ be a continuous process and define
\[
\mathcal{H}(Z):=
\sup_{t\in[0,1]}
\left|
\int_0^t \nabla U(r,Z_r)\Big(b(r,Z_r)-b\big(r,Z_{k_n(r)}\big)\Big)dr
\right|.
\]
Let $\bar X^n$ be the process in \eqref{4.1} and set
\[
\rho:=
\exp\!\left(
-\int_0^1 (\sigma^{-1}b)\big(r,\bar X^n_{k_n(r)}\big)dW_r
-\tfrac12\int_0^1
\big|(\sigma^{-1}b)\big(r,\bar X^n_{k_n(r)}\big)\big|^2dr
\right).
\]
Since $\sigma^{-1}b\in L^2([0,1];L^\infty(\mathbb{R}^d;\mathbb{R}^d))$, $\rho$ is a probability density.
By Girsanov's theorem and H\"older's inequality, for every $p\ge1$,
\begin{equation}\label{4.31}
\mathbb{E}\big[\mathcal{H}(X^n)^p\big]
=
\mathbb{E}\big[\rho\mathcal{H}(\bar X^n)^p\big]
\le
\big(\mathbb{E}\mathcal{H}(\bar X^n)^{2p}\big)^{\frac12}\big(\mathbb{E}\rho^2\big)^{\frac12}.
\end{equation}
Lemma~\ref{lem4.1} implies
\begin{equation}\label{4.32}
\|\mathcal{H}(\bar X^n)\|_{L^{2p}(\Omega)}\le Cn^{-\frac12}\log(n)^{\frac32}.
\end{equation}
Moreover, by an exponential martingale estimate, one gets $\mathbb{E}\rho^2\le C$.
Combining \eqref{4.30}--\eqref{4.32}, for every integer $n\ge 2$,  we obtain
\[
\Big\|\sup_{0\le r\le t}|X_r-X_r^n|\Big\|_{L^p(\Omega)}^2
\le
C\int_0^t f(r)\Big\|\sup_{0\le \tau\le r}|X_\tau-X_\tau^n|\Big\|_{L^p(\Omega)}^2dr
+
C n^{-1}\log(n)^3.
\]
An application of Gr\"onwall's inequality yields the strong error estimate \eqref{1.7}.

\section{Numerical experiments}\label{sec5}\setcounter{equation}{0}
We illustrate the strong convergence estimate in Theorem~\ref{the1.4} by numerical experiments.
We always consider the SDE \eqref{1.1} with $s=0$.
For a step number $n\in\mathbb{N}$, let $\Delta t=1/n$, $t_k=k\Delta t$ and
$\Delta W_k=W_{t_{k+1}}-W_{t_k}$ for $k=0,1,\ldots,n-1$.
Since the diffusion is bounded in time, in our examples we choose time-independent $\sigma(x)$.
Recall the polygonal-type scheme \eqref{1.4}:
\[
dX_t^n=b\big(t,X_{\kappa_n(t)}^n\big)dt+\sigma\big(X_{\kappa_n(t)}^n\big)dW_t,
\qquad \kappa_n(t)=\lfloor nt\rfloor/n.
\]
Integrating over $[t_k,t_{k+1}]$ yields the one-step update
\begin{equation}\label{5.1}
X_{t_{k+1}}^n
=X_{t_k}^n
+\int_{t_k}^{t_{k+1}} b\big(s,X_{t_k}^n\big)ds
+\sigma\big(X_{t_k}^n\big)\Delta W_k.
\end{equation}
Since $\Delta W_k$ is Gaussian with mean $0$ and covariance $\Delta t$, it can be sampled exactly by
\[
\Delta W_k \ \stackrel{d}{=}\ (\Delta t)^{\frac12} Z_k, \qquad Z_k\sim \mathcal{N}(0,I_d)\ \text{i.i.d.}
\]

We couple all discretizations on the same driving Wiener path and use a numerical reference solution
$X^{\mathrm{ref}}$ computed with the same scheme on a very fine grid $n_{\mathrm{ref}}=2^{18}$.
The pathwise supremum is approximated by the maximum over the reference grid nodes.
Monte Carlo uses $M=5000$ independent samples.
For $p\in\{2,4\}$ we measure the errors
\[
\cE_p^{\mathrm{sup}}(n)
=
\Big\|\sup_{0\le t\le1}|X_t^{\mathrm{ref}}-X_t^n|\Big\|_{L^p(\Omega)},
\qquad
\cE_p^{\mathrm{end}}(n)
=
\|X_1^{\mathrm{ref}}-X_1^n\|_{L^p(\Omega)},
\]
for $n\in\{64,128,256,512,1024,2048,4096,8192\}$.
We report two-level ratios (local rates)
\[
\mathrm{Rate(sup)}=\log_2\frac{\cE^{\mathrm{sup}}_{p}(n)}{\cE^{\mathrm{sup}}_{p}(2n)},
\qquad
\mathrm{Rate(end)}=\log_2\frac{\cE^{\mathrm{end}}_{p}(n)}{\cE^{\mathrm{end}}_{p}(2n)},
\]
as well as least-squares (LS) slopes of $\log_2\cE$ versus $\log_2 n$.

\subsection{Example A: one-dimensional Lebesgue--Dini drift}\label{ex5.1}
Define the Dini modulus
\[
\rho(r):=\bigl(\log(e/r)\bigr)^{-3},\qquad r\in(0,1].
\]
Let the sawtooth function $\varphi$ and the constants $a_k$, $k\ge 1$,
be as in Example~\ref{ex1.3}, i.e.
\[
\varphi(\upsilon):=\min\bigl\{\mathrm{dist}(\upsilon,\mathbb{Z}),1-\mathrm{dist}(\upsilon,\mathbb{Z})\bigr\},
\quad\mathrm{dist}(\upsilon,\mathbb{Z})=\inf_{m\in\mathbb{Z}}|\upsilon-m|, \quad
a_k:=\rho(2^{-k})-\rho(2^{-(k+1)}).
\]
Recall that $\varphi$ is bounded, piecewise linear,
and $1$-Lipschitz. The superposition of rescaled profiles $\varphi(2^k x)$ produces
multi-scale spatial oscillations, while the weights $a_k$ control their amplitude,
yielding a Dini-type modulus of continuity.

\smallskip
We consider the one-dimensional SDE
\begin{equation}\label{5.2}
dX_t=b(t,X_t)dt+\sigma(X_t)dW_t,\qquad X_0=0,
\end{equation}
with coefficients
\begin{equation}\label{5.3}
b(t,x)=f(t)g_{800}(x),\quad
\sigma(x)=1+0.5\tanh(x).
\end{equation}
where
\[
f(t)=t^{-\frac12}\bigl(\log(e/t)\bigr)^{-1},\quad
g_K(x)=\sum_{k=1}^{K} a_k\varphi(2^k x).
\]

The time factor $f(t)$ is critical at the $L^2$-threshold
up to a logarithmic correction and
\[
\int_0^1 |f(t)|^2\,dt=\int_0^1 \frac{dt}{t(\log(e/t))^2}<\infty.
\]
Thus,
\[
f\in L^2(0,1) \quad {\rm and} \quad f\notin L^q(0,1), \ \forall \  q>2.
\]

\smallskip
In particular, the drift employs the truncation level $K=800$, whereas a smaller value
$K=10$ is displayed in Figure~\ref{fig1} for illustration.
\begin{figure}[H]
  \centering
  \begin{subfigure}[H]{0.45\textwidth}
    \centering
    \includegraphics[width=\linewidth]{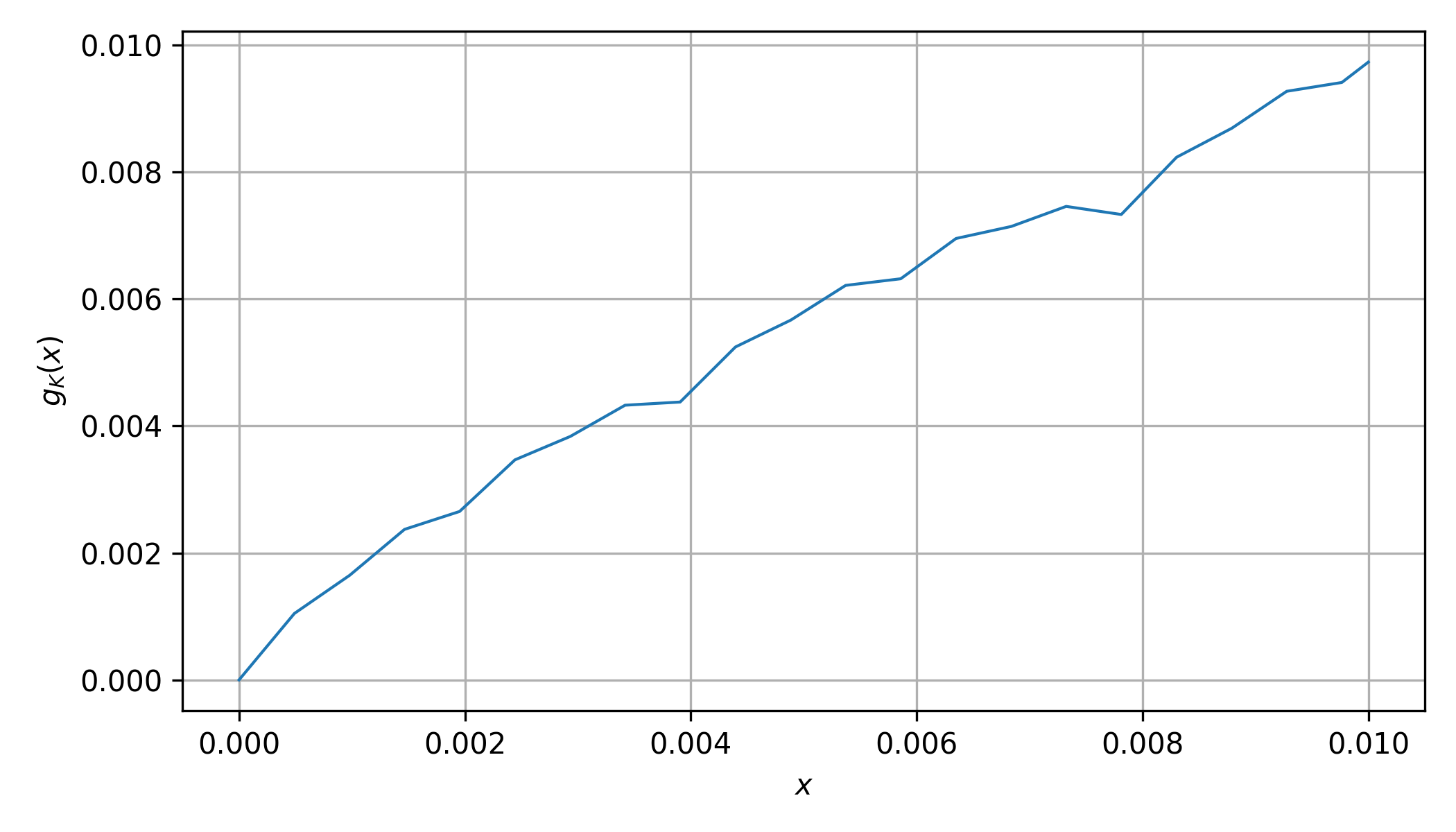}
    \caption{$K=10$}
  \end{subfigure}\hfill
 \begin{subfigure}[H]{0.45\textwidth}
   \centering
    \includegraphics[width=\linewidth]{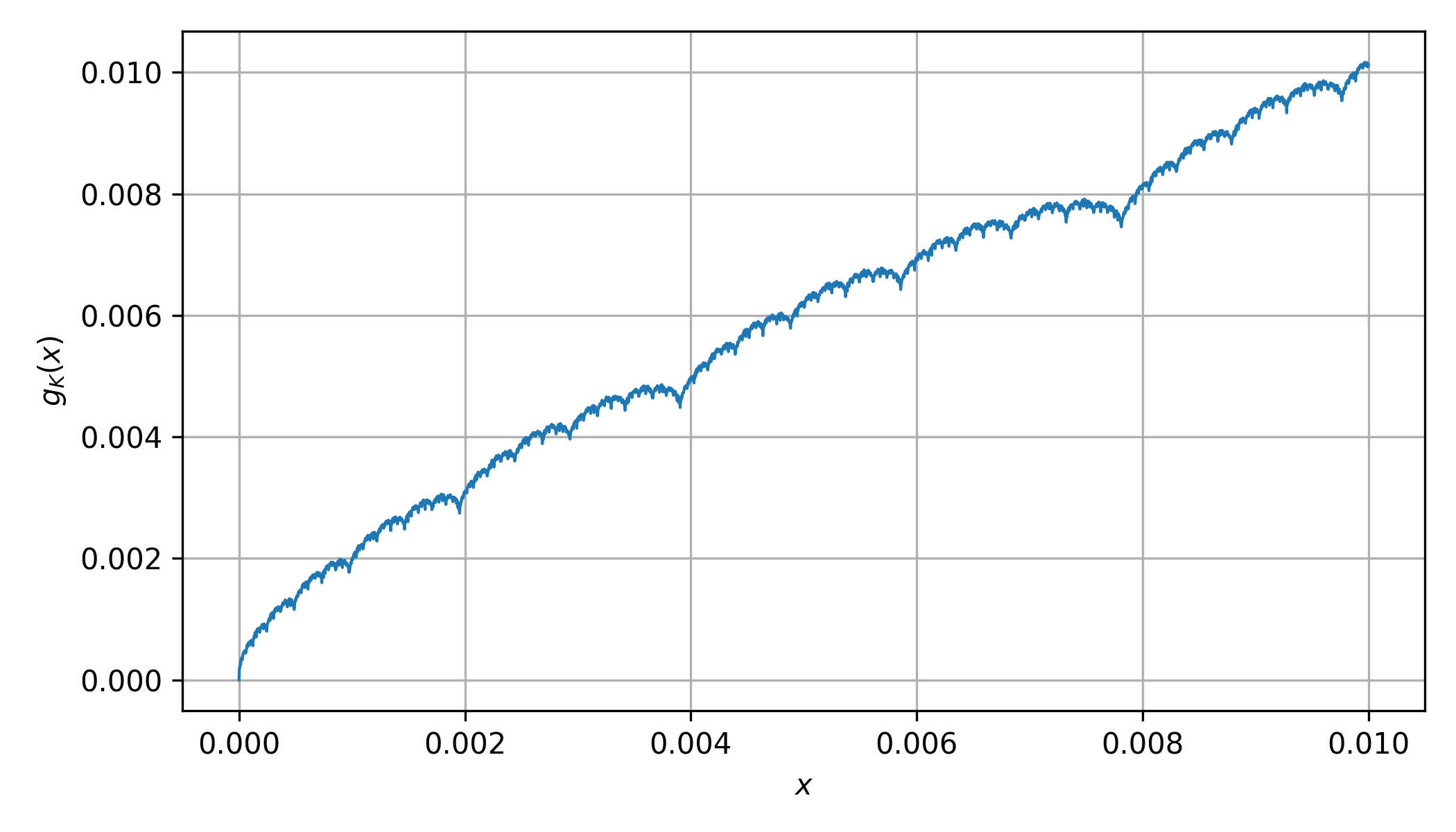}
    \caption{$K=800$}
 \end{subfigure}
  \caption{Spatial profiles of $g_K(x)=\sum_{k=1}^{K} a_k\varphi(2^k x)$ for different truncation levels $K$.}
  \label{fig1}
\end{figure}

Due to the truncation at $K=800$, the function $g_{800}$ in \eqref{5.3}
is globally Lipschitz (hence H\"older continuous of any order $\alpha\in(0,1]$),
with a Lipschitz constant bounded by $\sum_{k=1}^{800} 2^k a_k$.
While this truncation removes the genuine ``non-H\"older'' character of the
infinite series, the modulus of continuity on intermediate scales is still well
captured by the Dini-type modulus $\rho$, and the large effective Lipschitz
constant reflects pronounced multi-scale oscillations. Hence \eqref{5.2} serves
as a stringent numerical test in a low-regularity drift regime.

\smallskip
Since the drift is time-separable, $b(t,x)=f(t)g_{800}(x)$, the drift contribution
in the polygonal recursion \eqref{5.1} reduces to a deterministic weight
\begin{equation}\label{5.4}
\int_{t_k}^{t_{k+1}} b\bigl(s,X^n_{t_k}\bigr)ds
=\Bigl(\int_{t_k}^{t_{k+1}} f(s)\,ds\Bigr) g_{800}\bigl(X^n_{t_k}\bigr)
=: w_k\, g_{800}\bigl(X^n_{t_k}\bigr),
\end{equation}
where $w_k:=\int_{t_k}^{t_{k+1}} f(s)\,ds$
are computed once to high precision and reused across all Monte Carlo samples.
This avoids an additional time-quadrature bias induced by the borderline singularity
and ensures that the observed strong errors primarily reflect discretization effects
under low spatial regularity.

\begin{figure}[H]
  \centering
  \begin{subfigure}[h]{0.45\textwidth}
    \centering
    \includegraphics[width=\linewidth]{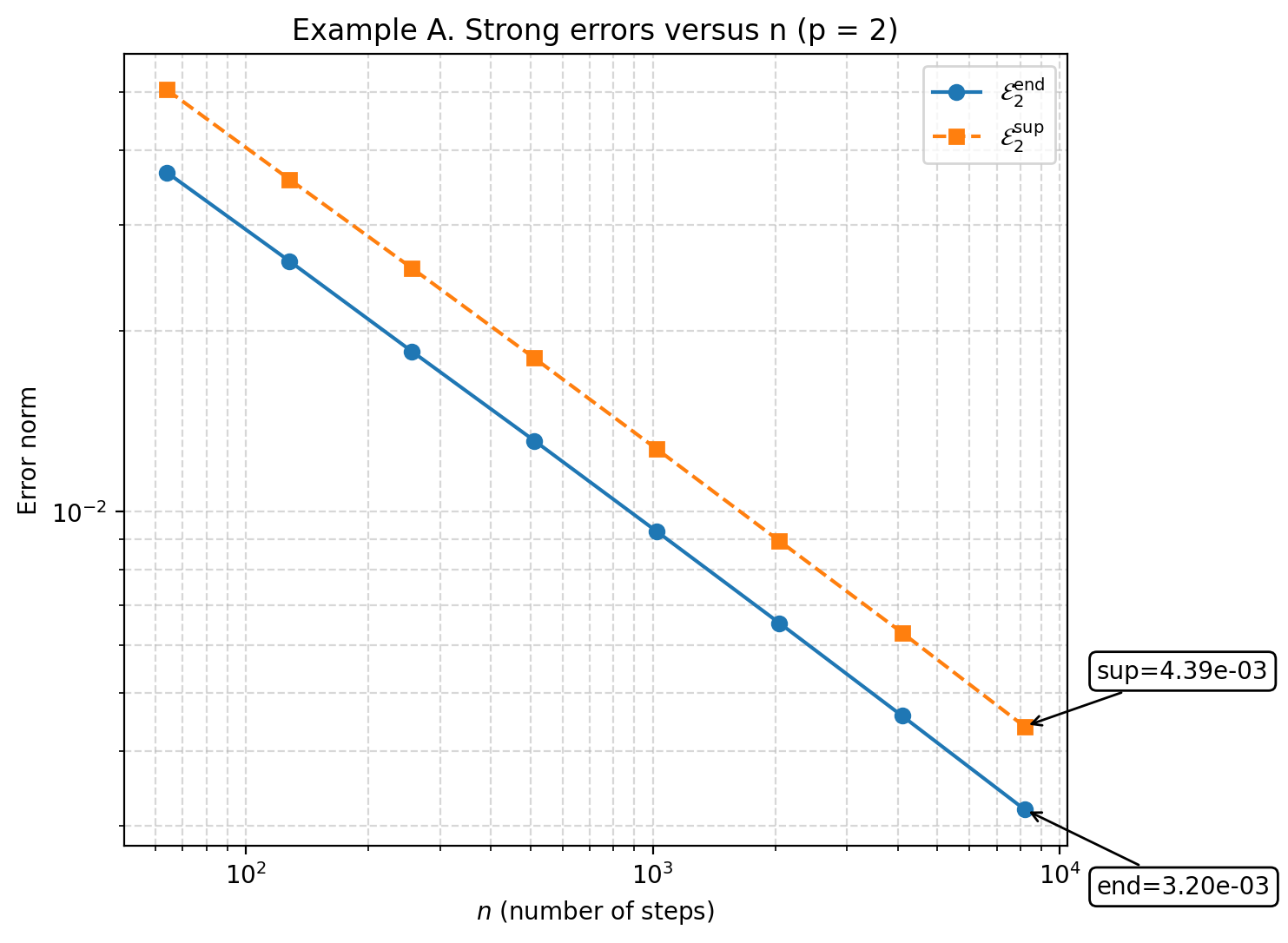}
    \caption{$L^2$}
  \end{subfigure}\hfill
  \begin{subfigure}[h]{0.45\textwidth}
    \centering
    \includegraphics[width=\linewidth]{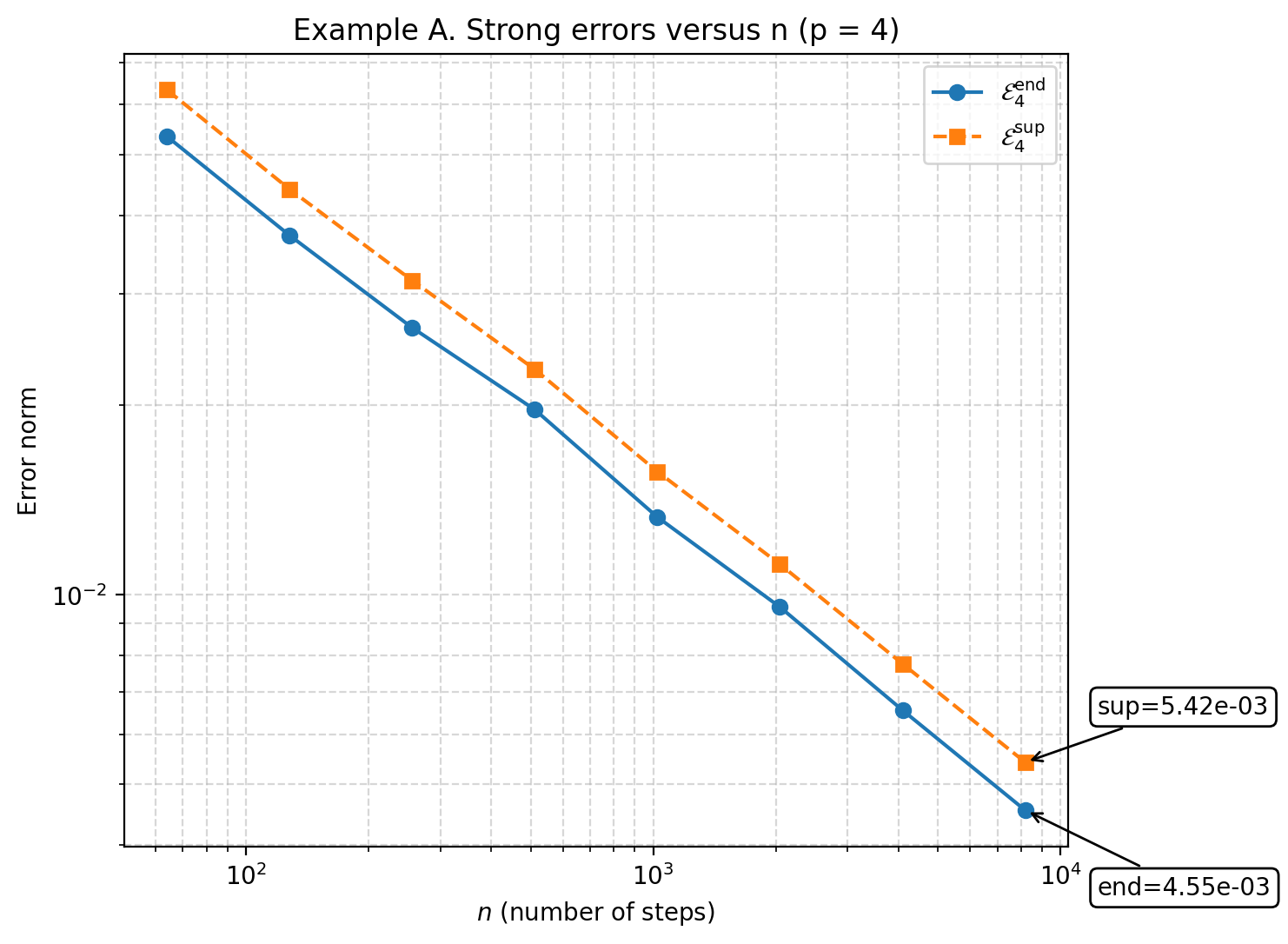}
    \caption{$L^4$}
  \end{subfigure}
  \caption{
  Example~A. Log--log plots of strong errors versus $n$
  (endpoint and pathwise supremum).}
   \label{fig:A}
\caption*{\footnotesize\textbf{Note.}
Log--log plots of the strong errors for \eqref{5.2} as functions of the time step
number $n$, including endpoint and pathwise supremum errors for $p=2$ (left)
and $p=4$ (right).
All discretizations are compared with the numerical reference solution.
The observed slopes are close to $1/2$.}
\end{figure}

\begin{table}[H]
\centering
\caption{Example~A. Strong errors and experimental rates (merged $L^2$ and $L^4$; common $n$)}.
\label{tab:A}
\small
\setlength{\tabcolsep}{5pt}
\resizebox{\linewidth}{!}{%
\begin{tabular}{ccccccccc}
\toprule
$n$ & \multicolumn{4}{c}{$L^2$} & \multicolumn{4}{c}{$L^4$}\\
\cmidrule(lr){2-5}\cmidrule(lr){6-9}
 & $\cE^{\text{end}}_2$ & $\cE^{\text{sup}}_2$ & Rate(end) & Rate(sup)
 & $\cE^{\text{end}}_4$ & $\cE^{\text{sup}}_4$ & Rate(end) & Rate(sup)\\
\midrule
64   & $3.67\times10^{-2}$ & $5.04\times10^{-2}$ & ---  & ---  &
        $5.34\times10^{-2}$ & $6.33\times10^{-2}$ & ---  & ---  \\
128  & $2.61\times10^{-2}$ & $3.58\times10^{-2}$ & $0.49$ & $0.49$ &
        $3.72\times10^{-2}$ & $4.40\times10^{-2}$ & $0.52$ & $0.52$ \\
256  & $1.85\times10^{-2}$ & $2.54\times10^{-2}$ & $0.50$ & $0.49$ &
        $2.65\times10^{-2}$ & $3.15\times10^{-2}$ & $0.49$ & $0.48$ \\
512  & $1.31\times10^{-2}$ & $1.80\times10^{-2}$ & $0.49$ & $0.50$ &
        $1.97\times10^{-2}$ & $2.28\times10^{-2}$ & $0.43$ & $0.47$ \\
1024 & $9.27\times10^{-3}$ & $1.27\times10^{-2}$ & $0.50$ & $0.50$ &
        $1.33\times10^{-2}$ & $1.57\times10^{-2}$ & $0.57$ & $0.54$ \\
2048 & $6.53\times10^{-3}$ & $8.94\times10^{-3}$ & $0.51$ & $0.51$ &
        $9.55\times10^{-3}$ & $1.12\times10^{-2}$ & $0.47$ & $0.49$ \\
4096 & $4.58\times10^{-3}$ & $6.29\times10^{-3}$ & $0.51$ & $0.51$ &
        $6.55\times10^{-3}$ & $7.76\times10^{-3}$ & $0.54$ & $0.52$ \\
8192 & $3.20\times10^{-3}$ & $4.39\times10^{-3}$ & $0.52$ & $0.52$ &
        $4.55\times10^{-3}$ & $5.42\times10^{-3}$ & $0.53$ & $0.52$ \\
\midrule
\multicolumn{3}{r}{LS slope ($L^2$):} & $0.51$ & $0.51$
& \multicolumn{2}{r}{LS slope ($L^4$):} & $0.52$ & $0.51$\\
\bottomrule
\end{tabular}}
\caption*{\footnotesize\textbf{Note.} Strong endpoint errors and pathwise supremum errors for the one-dimensional example
under different time resolutions $n$.
Two-level ratios provide empirical convergence rates typically close to $1/2$ in both $L^2$ and $L^4$ norms.
LS slopes computed over the finest levels summarize the overall strong convergence behavior.}
\end{table}

\subsection{Example B: two-dimensional coupled drift and diffusion}\label{subsec5.2}
We set $d=2$ and consider the two-dimensional SDE
\begin{equation}\label{5.5}
dX_t=b(t,X_t)\,dt+\sigma(X_t)\,dW_t,\qquad X_0=(0,0),
\end{equation}
where $W_t$ is a  two-dimensional standard Wiener process. We keep the functions $f$ and $g$ from Example~A and define the coupled drift
\[
b(t,x)= f(t)\,G(x)+H(x),
\]
where $x=(x_1,x_2)\in\mathbb{R}^2$ and
\[
G(x)=
\begin{pmatrix}
g_{800}(x_1+0.35\,x_2)\\[1mm]
g_{800}(x_2-0.25\,x_1)
\end{pmatrix}.
\]
Each component of $G$ depends on both spatial variables, yielding a genuinely
two-dimensional coupling.
The additional bounded nonlinear term $H$ is given by
\[
H(x)=
\begin{pmatrix}
0.25\tanh(x_2)+0.1\psi(x_1-x_2)\\[1mm]
\psi(x_2)+0.12\tanh(x_1)+0.08\psi(x_1+x_2)
\end{pmatrix}
\quad {\rm and} \quad
\psi(z)=\mathrm{sign}(z)\frac{|z|^{0.4}}{1+|z|^{0.4}}.
\]

We further choose a non-diagonal diffusion coefficient of the form
\[
\sigma(x)= I + \varepsilon S(x), \qquad \varepsilon=0.25,
\]
with
\[
S(x)=
\begin{pmatrix}
\tanh(x_1) & 0.3\tanh(x_1+x_2)\\
0.3\tanh(x_1+x_2) & \tanh(x_2)
\end{pmatrix}.
\]
This yields a bounded, symmetric, and non-diagonal diffusion matrix,
ensuring genuine noise coupling while preserving uniform ellipticity.

\smallskip
In the polygonal recursion~\eqref{5.1}, the time-singular components of the drift
are integrated exactly via the precomputed weights
$w_k=\int_{t_k}^{t_{k+1}} f(s)ds$ as in~\eqref{5.4}, while the remaining drift
terms are evaluated explicitly at the left endpoints.

\begin{figure}[H]
\centering
\begin{subfigure}[H]{0.45\textwidth}
  \centering
  \includegraphics[width=\linewidth]{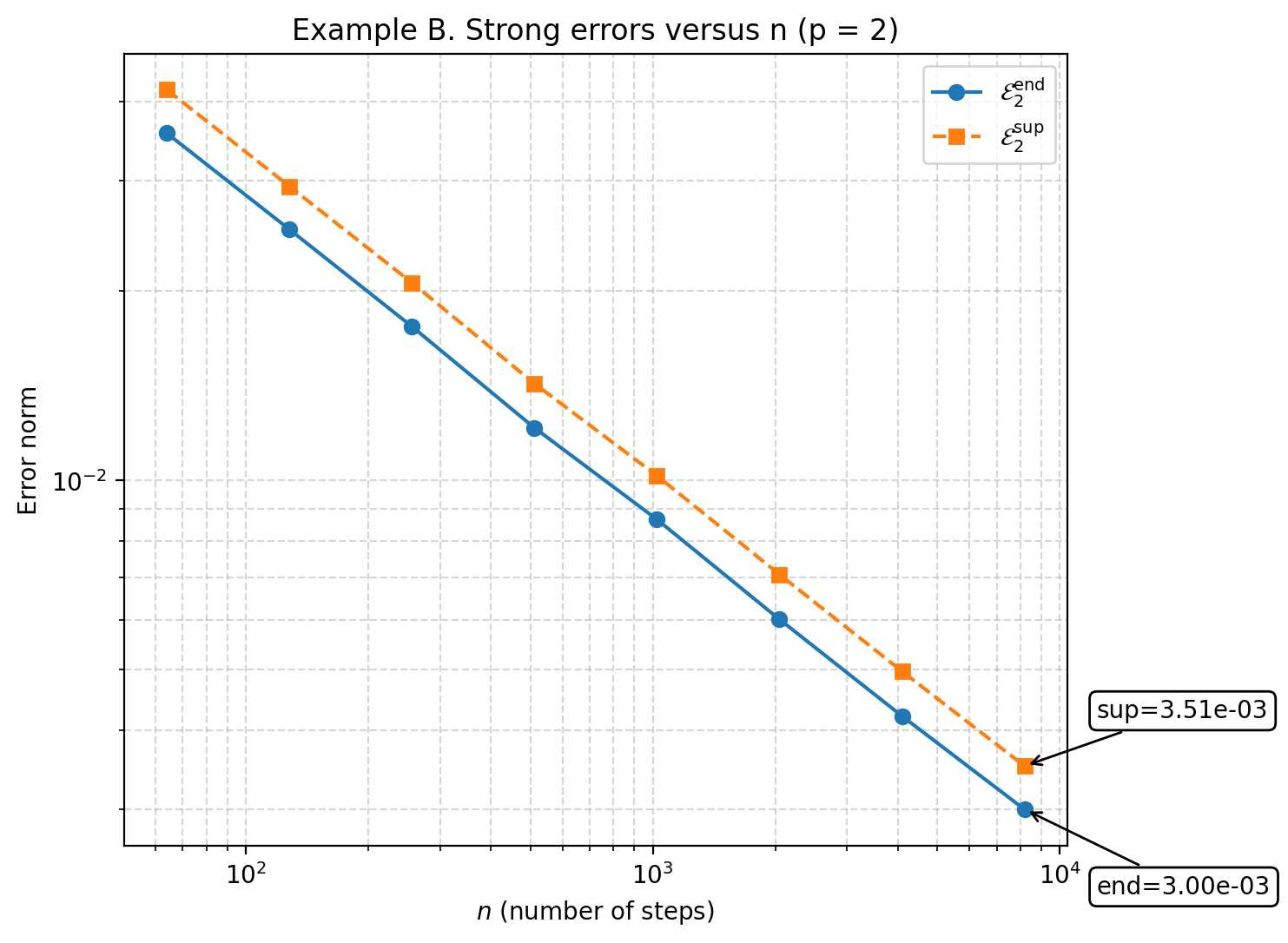}
 \caption{$L^2$}
\end{subfigure}\hfill
\begin{subfigure}[H]{0.45\textwidth}
  \centering
  \includegraphics[width=\linewidth]{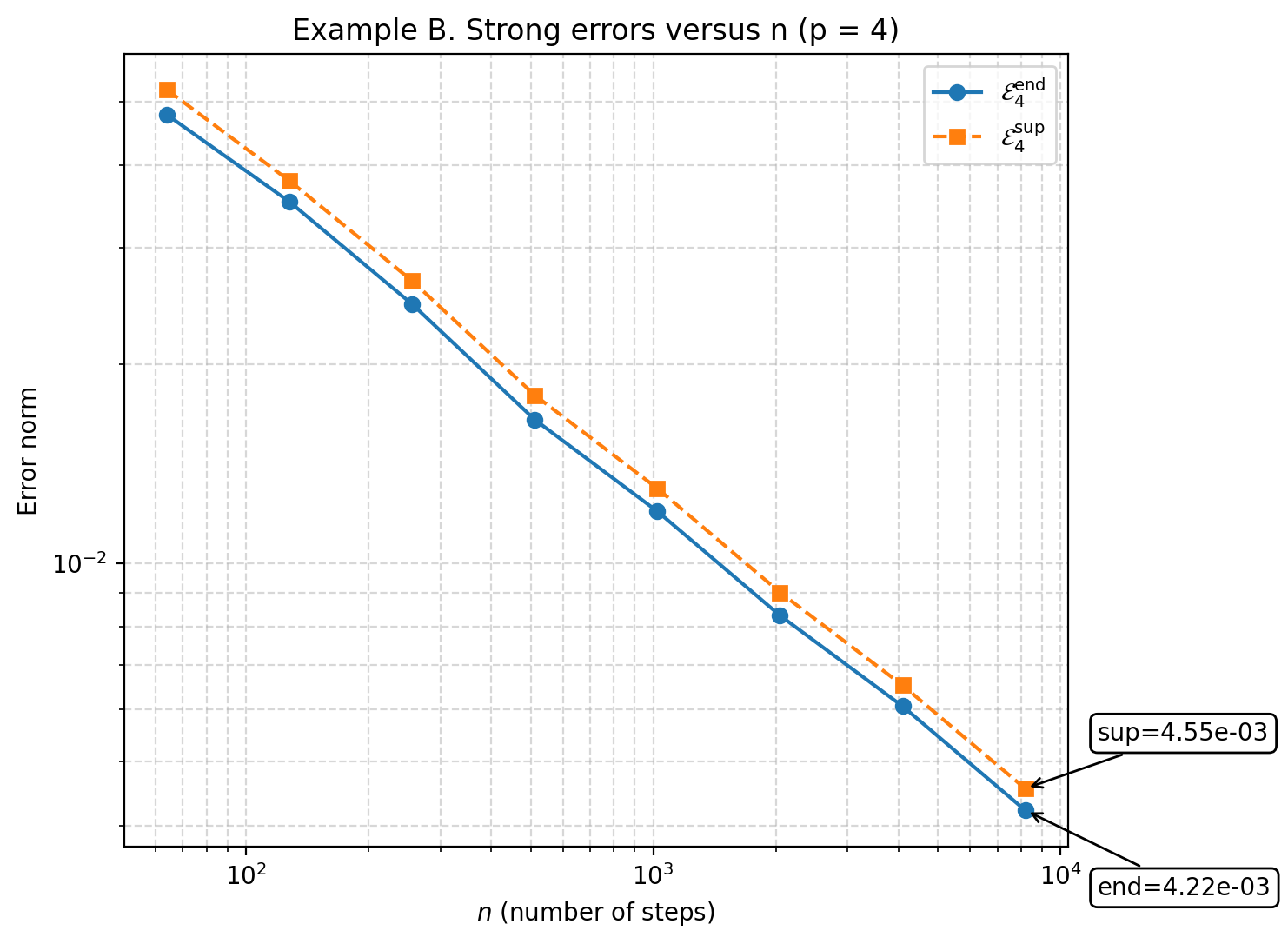}
  \caption{$L^4$}
\end{subfigure}
\caption{Example~B. Log--log plot of strong errors versus $n$
(endpoint and supremum).}
\label{fig:B}
\caption*{\footnotesize\textbf{Note.}
Log--log plots of the strong errors for the two-dimensional coupled example
as functions of the time step number $n$, including endpoint and pathwise
supremum errors for $p=2$ (left) and $p=4$ (right).
All discretizations are compared with the numerical reference solution.
The slopes indicate a strong convergence rate close to $1/2$.}
\end{figure}

\begin{table}[H]
\centering
\caption{Example~B. Strong errors and experimental rates (merged $L^2$ and $L^4$; common $n$).}
\label{tab:B}
\small
\setlength{\tabcolsep}{5pt}
\resizebox{\linewidth}{!}{%
\begin{tabular}{ccccccccc}
\toprule
$n$ & \multicolumn{4}{c}{$L^2$} & \multicolumn{4}{c}{$L^4$}\\
\cmidrule(lr){2-5}\cmidrule(lr){6-9}
 & $\cE^{\text{end}}_2$ & $\cE^{\text{sup}}_2$ & Rate(end) & Rate(sup) & $\cE^{\text{end}}_4$ & $\cE^{\text{sup}}_4$ & Rate(end) & Rate(sup)\\
\midrule
64 & $3.57\times10^{-2}$ & $4.18\times10^{-2}$ & --- & --- & $4.77\times10^{-2}$ & $5.21\times10^{-2}$ & --- & --- \\
128 & $2.51\times10^{-2}$ & $2.94\times10^{-2}$ & 0.51 & 0.51 & $3.52\times10^{-2}$ & $3.79\times10^{-2}$ & 0.44 & 0.46 \\
256 & $1.76\times10^{-2}$ & $2.06\times10^{-2}$ & 0.51 & 0.51 & $2.47\times10^{-2}$ & $2.68\times10^{-2}$ & 0.51 & 0.50 \\
512 & $1.21\times10^{-2}$ & $1.43\times10^{-2}$ & 0.54 & 0.53 & $1.65\times10^{-2}$ & $1.80\times10^{-2}$ & 0.58 & 0.58 \\
1024 & $8.67\times10^{-3}$ & $1.02\times10^{-2}$ & 0.48 & 0.49 & $1.20\times10^{-2}$ & $1.30\times10^{-2}$ & 0.46 & 0.47 \\
2048 & $6.01\times10^{-3}$ & $7.09\times10^{-3}$ & 0.53 & 0.52 & $8.33\times10^{-3}$ & $9.03\times10^{-3}$ & 0.52 & 0.52 \\
4096 & $4.22\times10^{-3}$ & $4.97\times10^{-3}$ & 0.51 & 0.51 & $6.07\times10^{-3}$ & $6.53\times10^{-3}$ & 0.46 & 0.47 \\
8192 & $3.00\times10^{-3}$ & $3.51\times10^{-3}$ & 0.49 & 0.50 & $4.22\times10^{-3}$ & $4.55\times10^{-3}$ & 0.52 & 0.52 \\
\midrule
\multicolumn{3}{r}{LS slope ($L^2$):} & $0.51$ & $0.51$
& \multicolumn{2}{r}{LS slope ($L^4$):} & $0.50$ & $0.50$\\
\bottomrule
\end{tabular}}
\caption*{\footnotesize\textbf{Note.}
Strong endpoint errors and pathwise supremum errors for the two-dimensional coupled example
under different time resolutions $n$.
Two-level ratios provide empirical convergence rates typically close to $1/2$ in both $L^2$ and $L^4$ norms.
LS slopes computed over the finest levels summarize the overall strong convergence behavior.}
\end{table}

\noindent\textbf{Conclusion of the numerical experiments.}
The numerical experiments in Examples~A and~B illustrate the strong convergence
behavior of the polygonal Euler--Maruyama scheme under low-regularity drift
assumptions.
In both cases, the endpoint and pathwise supremum errors decay approximately
linearly on the log--log scale as the time step is refined, with empirical
convergence rates close to $1/2$, in agreement with the estimate of
Theorem~\ref{the1.4}.

\smallskip
Example~A demonstrates this behavior in a one-dimensional setting with a
borderline time-singular drift that remains square-integrable on $(0,1)$ and is
only Dini continuous in space.
Example~B extends the observation to a genuinely coupled two-dimensional system
with non-diagonal diffusion coefficients.
Overall, the numerical results provide supporting evidence for the robustness of
the proposed scheme in the presence of Lebesgue--Dini type drifts and uniformly
elliptic, non-constant diffusion coefficients.

\section*{Acknowledgements} Guangying Lv was partially supported by the National Natural Science Foundation of China grant 12171247. Junhao Hu was partially supported by the National Natural Science Foundation of China grant 62373383 and the Fundamental Research Funds for the Central Universities of South-Central Minzu University grants KTZ20051 and CZT20020.

\section*{Conflicts of Interest} The authors declare that they have no competing interests.

\section*{Data Availability Statements} Data sharing is not applicable to this article as no new data were created or analyzed in this study.

\end{document}